\documentclass{amsart}
\usepackage{amscd,amssymb,amsopn,amsmath,amsthm,graphics,amsfonts,accents,enumerate,verbatim,calc}
\usepackage[dvips]{graphicx}
\usepackage[colorlinks=true,linkcolor=red,citecolor=blue]{hyperref}
\usepackage[all]{xy}
\usepackage{tabu}
\usepackage{tikz}
\usepackage{mathdots}

\calclayout

\newcommand{\rt}{\rightarrow}
\newcommand{\lrt}{\longrightarrow}

\newcommand{\st}{\stackrel}

\newcommand{\La}{\Lambda}
\newcommand{\Ga}{\Gamma}

\newcommand{\CA}{\mathcal{A} }

\newcommand{\CF}{\mathcal{F} }
\newcommand{\CG}{\mathcal{G} }

\newcommand{\CP}{\mathcal{P} }
\newcommand{\CQ}{\mathcal{Q} }

\newcommand{\CS}{\mathcal{S} }
\newcommand{\CT}{\mathcal{T} }
\newcommand{\CX}{\mathcal{X} }
\newcommand{\CY}{\mathcal{Y} }

\newcommand{\CV}{\mathcal{V}}
\newcommand{\CU}{\mathcal{U}}

\newcommand{\CB}{\mathcal{B} }

\newcommand{\Mod}{{\rm{Mod\mbox{-}}}}

\newcommand{\mmod}{{\rm{{mod\mbox{-}}}}}

\newcommand{\prj}{{\rm{prj}\mbox{-}}}
\newcommand{\Ind}{\rm{{Ind}\mbox{-}}}

\newcommand{\add}{{\rm{add}\mbox{-}}}

\newcommand{\Hom}{{\rm{Hom}}}
\newcommand{\Ext}{{\rm{Ext}}}
\newcommand{\End}{{\rm{End}}}

\newcommand{\bsm}{\begin{smallmatrix}}
	\newcommand{\esm}{\end{smallmatrix}}
\newcommand{\bbm}{\begin{matrix}}
	\newcommand{\ebm}{\end{matrix}}

\theoremstyle{plain}
\newtheorem{theorem}{Theorem}[section]
\newtheorem{corollary}[theorem]{Corollary}
\newtheorem{lemma}[theorem]{Lemma}

\newtheorem{proposition}[theorem]{Proposition}

\theoremstyle{definition}
\newtheorem{definition}[theorem]{Definition}
\newtheorem{example}[theorem]{Example}

\newtheorem{construction}[theorem]{Construction}

\newtheorem{remark}[theorem]{Remark}
\newtheorem{setup}[theorem]{Setup}

\theoremstyle{plain}

\theoremstyle{definition}

\numberwithin{equation}{section}

\begin{document}

\title[From subcategories to the entire module categories ]{From subcategories to the entire module categories}

\author[Rasool Hafezi]{Rasool Hafezi }

\address{School of Mathematics, Institute for Research in Fundamental Sciences (IPM), P.O.Box: 19395-5746, Tehran, Iran}
\email{hafezi@ipm.ir}




\begin{abstract}
In this paper we  show that  how the representation theory of  subcategories (of the category of modules over an Artin algebra)  can be connected to the representation theory of all modules over some algebra. The subcategories  dealing with are some certain subcategories of the morphism categories (including submodule categories studied recently by Ringel and  Schmidmeier) and of the  Gorenstein projective modules over (relative) stable Auslander algebras. These two kinds of subcategories, as will be seen, are  closely related to each other. To make such a connection, we  will  define  a functor from each type of  the subcategories to the  category of modules over some Artin algebra. It is shown that to compute the almost split sequences in the subcategories it is enough to do the computation  with help of  the corresponding functors in the  category of modules over  some Artin algebra which is known and easier to work. Then as an application the most part  of  Auslander-Reiten quiver of the subcategories is obtained only  by  the Ausalander-Reiten quiver of  an appropriate algebra and next     adding the remaining vertices and arrows in an obvious way. As a special case, when $\La$ is a Gorenstein Artin algebra of finite representation type, then  the subcategories of Gorenstein projective modules over the $2 \times 2$ lower triangular matrix algebra over $\La$ and the stable Auslander algebra of $\La$  can be estimated by the  category of modules  over the stable Cohen-Macaulay Auslander algebra of $\La$.
\end{abstract}

\maketitle
\section{Introduction}
The most basic problem in the  representation theory  of Artin algebras is to classify the indecomposable finitely generated modules. Since the general problem is known to be very difficult, special attention might be paid to some well-behaved subcategories. In this paper we have plane to investigate simultaneously  the subcategory of Gorenstein projective modules over the relative stable Auslander algebras and some certain subcategories of the morphism category. As will be discussed in below, there is a nice relationship between 
 these two types of subcategories considered  in our paper. The notion of Gorenstein projective modules was first defined by Maurice Auslander in the mid-sixties, see \cite{AB}, for  finitely generated projective  $\La$-modules over a Noetherian ring, called $G$-dimension zero. Later  they have found important applications in commutative algebra, algebraic geometry, singularity theory and relative homological algebra. Enochs and Jenda in \cite{EJ} generalized the notion of Gorenstein projective modules to (not necessarily finitely generated) modules over any ring $R$, under the name of Gorenstein projective modules. Studying of (finitely generated)  Gorenstein projective modules  has attracted  attention in the setting of Artin algebras, see for instance \cite{AR2, AR3, CSZ, RZ3, RZ4} and etc. More recently, Ringel and Zhang in \cite{RZ2}, was shown that  there is a nice bijection between the indecomposable modules over  a finite dimensional hereditary algebra over an algebraic closed filed and the indecomposable ones of the stable category of Gorenstein projective modules over the tensor algebra of the algebra of dual numbers and the given hereditary  algebra.\\
  Another subcategories we are interested in this paper are the subcategories of  the morphism category $H(\La)$ of Artin algebra $\La$ arising by a quasi-resolving subcategory $\CX$ of the category $\mmod \La$ of finitely generated  right modules over $\La$.   Denote by $\CS_{\CX}(\La)$ the subcategory of $H(\La)$ consisting of those monomorphisms in $\mmod \La$ with terms in  $\CX$ such that whose cokernels belong to $\CX.$ When $\CX=\mmod \La$, the subcategory $\CS_{\mmod \La}(\La)$, simply $\CS(\La)$, becomes the submodule category, which have been studied  extensively by Ringel and  Schmidmeier in \cite{RS1, RS2}, and also a generalization of their
 works given in \cite{LZ2, XZ, XZZ}. In addition,  a surprising link, established in \cite{KLM},   between the stable submodule category with the singularity theory via weighted projective lines of type $(2,3,p)$ is discovered. It is  worth noting that when $\CX$ is equal to the subcategory $\rm{Gprj}\mbox{-}\La$ of Gorenstein projective $\La$-modules, then $\CS_{\rm{Gprj}\mbox{-}\La}(\La)$ is nothing else than the subcategory of Gorenstein projective modules over $T_2(\La)$, lower triangular $2 \times 2$ matrix algebra over $\La,$  here $\mmod T_2(\La)$ is identified by $H(\La)$.
 
Our main purpose in this paper is to study   the Auslander-Reiten theory, as a powerful tool in the modern representation theory, over the above-mentioned subcategories by approaching with modules category over some Artin algebras. The Auslander-Reiten theory for the module categories over an Artin algebra is much more known than some subcategories. At least, the Auslander-Retien translation in the module categories is computed, that is, $D\rm{Tr}$. Such a connection can be helpful to transfer those known results to the subcategories. This aim is inspired by the work of \cite{RZ1} as follows: In \cite{RZ1} the authors make a connection between $\CS(k[x]/(x^n))$ and the module category $\mmod \Pi_{n-1}$ of the preprojective algebra $\Pi_{n-1}$ of type $A_{n-1}$ via defining two functors  which originally come from the  works of Auslander-Reiten and Li-Zhang. We know that $\Pi_{n-1}$ is the stable Auslander algebra of representation-finite algebra $k[x]/(x^n)$ (see \cite{DR}, Theorem 3, Theorem 4 and Chapter 7). Hence those functors appeared in \cite{RZ1} are indeed functors from $\CS(k[x]/(x^n))$ to $\rm{Aus}(\underline{\rm{mod}}\mbox{-}k[x]/(x^n))$. We recall that the stable Auslander algebra $\rm{Aus}(\underline{\rm{mod}}\mbox{-}\La)$ of a representation-finite algebra $\La$ is the endomorphism algebra $\underline{\End}_{\La}(M)$  in the sable category $\underline{\rm{mod}}\mbox{-}\La$,  where $M$ is  a  representation generator  for $\La$, i.e., the additive closure $\rm{add}\mbox{-}M$ of $M$ in $\mmod \La$ is equal to $\mmod \La.$ Recall  for an additive category $\mathcal{C}$, $\mmod \mathcal{C}$ denotes the category of finitely presented functors over $\mathcal{C}$. It is known that in the case that $\mathcal{C}$ is generated by an object $C$, i.e., $\rm{add}\mbox{-}C=\mathcal{C}$, then the evaluation functor on  $C$ induces an  equivalence of the categories $\mmod \mathcal{C}\simeq \mmod \End_{\mathcal{C}}(C).$  The  observation given in \cite{RZ1} motivates us to define a functor from $\CS_{\CX}(\La)$ to $\mmod \underline{\CX}$, denote by $\Psi_{\CX}$, see Construction \ref{FirstCoonstr}, for a  a quasi-resolving subcategory $\CX$ of $\mmod \La.$  Here $\underline{\CX}$ is the stable  category of $\CX$, that is, the quotient of $\CX$ by the ideal of morphisms factoring through a projective module. The functor $\Psi_{\CX}$ is a relative version of one of the two  functors   studied  in \cite{RZ1}, and also \cite{E},  in a   functorial language. The functor $\Psi_{\CX}$ induces an equivalence between the  quotient category $\CS_{\CX}(\La)/\CU$ and $  \mmod \underline{\CX}$, where $\CU$ is the ideal generated by the objects in $\CS_{\CX}(\La)$ of the form  $(X \st{1}\rt X )$ and $ (0 \rt X)$, where $X$ runs through objects in $\mathcal{X}$ (see Theorem \ref{Thefirst}). With some additional conditions on $\CX$, see Setup \ref{Setup1}, we will observe in Section \ref{Section 5} the almost split sequences in $\CS_{\CX}(\La)$ is preserved by $\Psi_{\CX}.$  Next, since the notion of the Auslander-Reiten quivers is based on the almost split sequences, by our result we observe  that the Auslander-Reiten quiver $\Gamma_{\underline{\CX}}$ of $\mmod \underline{\CX}$  can be considered as a valued full subquiver of $\Gamma_{\CS_{\CX}(\La)}$ such that contains all vertices in $\Gamma_{\CS_{\CX}(\La)}$ except those of vertices corresponding to the isomorphism classes of the  indecomposable objects having of the form  $(X\st{1}\rt X)$ and $(0\rt X)$, see Theorem \ref{main4}. Therefore, the task to find the full of $\Gamma_{\CS_{\CX}(\La)}$ is only to add the remaining vertices  corresponding to the  isomorphism classes of  the indecomposable objects in $\CS_{\CX}(\La)$  having the simplest structure among of  the other indecomposable objects in $\CS_{\CX}(\La),$  and arrows attached to them. In the case that $\CX$ is of finite representation type, i.e., $\CX=\rm{add}\mbox{-}X$ for some $X \in \CX,$ we have $\mmod \CX\simeq \mmod \rm{Aus}(\underline{\CX})$, where $\rm{Aus}(\underline{\CX})=\underline{End}_{\La}(X)$. The algebra $\rm{Aus}(\underline{\CX})$ is called the relative stable Auslander algebra with respect to the subcategory $\CX$.  Therefore, in this way, the representation theoretic properties of the subcategory  $\CS_{\CX}(\La)$ via the functor $\Psi_{\CX}$ can be approached by the module category over the associated relative stable Auslander algebra $\rm{Aus}(\underline{\CX})$. 

In parallel with  the subcategories of the form $\CS_{\CX}(\La)$, we are also interested  in considering the subcategory $\rm{Gprj}\mbox{-}\underline{\CX}$ of Gorenstein projective functors in   the category $\mmod \underline{\CX}$ of  finitely presented functors over $\underline{\CX}$. By the way, as shown in $(\dagger)$, the second type of the subcategories, considered in the paper, is closely related  to the first type of the  subcategories, have been already discussed. Let $\CX $ be a quasi-resolving subcategory of $\mmod \La$ as in Setup \ref{4.17}.  Denote $\CY=\CX\cap \rm{Gprj}\mbox{-}\La$. By a characterization of  Gorenstein projective functors in $\mmod \underline{\CX}$ via  their projective resolutions in $\mmod \CX$ given in  Theorem \ref{Classification of Gorenstein projective functors}, we construct our second important functor   ${}_{\underline{\CY}}\hat{\varUpsilon}_{\underline{\CX}}:\mmod \underline{\CY}\rt \rm{Gprj}\mbox{-}\underline{\CX}$, where $\rm{Gprj}\mbox{-}\underline{\CX}$ denotes the subcategory of Gorenstein projective functors in $\mmod \underline{\CX}$. That is fully faithful exact. An important property of the functor ${}_{\underline{\CY}}\hat{\varUpsilon}_{\underline{\CX}}$
is that  the  almost split sequences in $\mmod \underline{\CY}$ are   mapped into the  almost split sequences in  $\rm{Gprj}\mbox{-}\underline{\CX}$ (Proposition \ref{AlmsotsplitXY}). Hence, as we are looking for, the $\mmod \underline{\CY}$ can be viewed as suitable candidate to study the subcategory $\rm{Gprj}\mbox{-}\underline{\CX}$. Because $\CY$ satisfies the required conditions in the first part, then we have:
$$\CS_{\CY}(\La)\st{\Psi_{\CY}}\rt \mmod \underline{\CY}\st{{}_{\underline{\CY}}\hat{\varUpsilon}_{\underline{\CX}}}\rt \rm{Gprj}\mbox{-}\underline{\CX} \ \ \ \ \ \ \ \ \ \ (\dagger)$$
such that $\Psi_{\CY}$ is full, dense and objective, and ${}_{\underline{\CY}}\hat{\varUpsilon}_{\underline{\CX}}$ is fully faithful, exact, and  almost dense (see Definition \ref{almsotdense}) if assume $\CX$ is of finite representation type. Moreover, each of the functors behave well with the almost split sequences.

  Specializing our two approaches in the above to Gorenstein algebras,  we reach the following interesting corollary.

 \begin{corollary}
 	Let $\La$ be a Gorenstein algebra of finite representation type. Let $A=\rm{Aus}(\underline{\rm{mod}}\mbox{-}\La)$ and $B=\rm{Aus}(\underline{\rm{Gprj}}\mbox{-}\La)$ be, respectively, the stable Auslander algebra and stable Cohen-Macaulay  Auslander algebra  of $\La$.
 	Set ${}_{\CG}\Gamma_{\La}:={}_{\underline{\rm{Gprj}}\mbox{-}\La}\hat{\varUpsilon}_{\underline{\rm{mod}}\mbox{-}\La} $ and $\Psi_{\CG}:=\Psi_{\rm{Gprj}\mbox{-}\La}$.  Then we have the following assertions. 
 	\begin{itemize}
 		\item [$(1)$] There are  the following  functors
 		$$\rm{Gprj}\mbox{-}T_2(\La) \st{\Psi_{\CG}}\lrt \mmod B\st{{}_{\CG}\Gamma_{\La}} \rt \rm{Gprj}\mbox{-} A,$$
 		where the functor $\Psi_{\CG}$ is full, dense and objective, and the functor  ${}_{\CG}\Gamma_{\La}$ is fully faithful, almost dense and exact. Moreover, each of the functors ${}_{\CG}\Gamma_{\La}, \Psi_{\CG}$ preservers the almost split sequences which do not end at either $(0\rt X), (X\st{1}\rt X)$ or $(\Omega_{\La}(X)\hookrightarrow P)$, where $X$ is a non-projective indecomposable module in $\mmod \La$;
  		\item [$(2)$] The Auslander-Reiten quiver $\Gamma_{B}$  is a full valued  subquiver of the Auslander-Reiten quivers  $\Gamma_{\rm{Gprj}\mbox{-}T_2(\La)}$  and $\Gamma_{\rm{Gprj}\mbox{-} A}$, i.e.,
  		$$\Gamma_{\rm{Gprj}\mbox{-}T_2(\La)}\hookleftarrow \Gamma_{B }\hookrightarrow \Gamma_{\rm{Gprj}\mbox{-} A}.$$  		
 	\end{itemize}
 \end{corollary}
Since $\rm{Gprj}\mbox{-}\La$ is a triangulated category,  $\rm{Aus}(\underline{\rm{Gprj}}\mbox{-}\La)$ is a self-injective algebra, assuming $\rm{Gprj}\mbox{-}\La$ has only finitely many indecomposable modules up to isomorphisms. Based on the above observation,  the study of  the subcategory of Gorenstein projective modules over $T_2(\La)$ (and $\rm{Aus}(\underline{mod}\mbox{-}\La)$ if $\La$ is of finite representation type) is connected to module category over a self-injective algebra.  The class of self-injective algebras is one of the important  classes in the representation theory of Artin algebras which are well-understood. Hence the knowledge concerning  the class of self-injective algebras by the above observation  may carry over into the  subcategory of Gorenstein  projective modules over $T_2(\La)$, and also over $\rm{Aus}(\underline{\rm{mod}}\mbox{-}\La)$ where  $\La$ is assumed to be of finite representation type. 

The paper is organized as follows. In Section 2,  some important facts and notions, including functor categories, Gorenstein projective modules and triangular matrix algebras,  which will be needed  in the paper are recalled.  In Section 3, we will  define the functor $\Psi_{\CX}$ and  Theorem \ref{Thefirst} will be proved. In Section 4, firstly, an explicit description of $n$-th syzygies of the functors in $\mmod \underline{\CX}$ is given in terms of their projective resolutions in $\mmod \CX$, see Corollary \ref{N-the syzygy}. Secondly, we will use the description of  $n$-th syzygies to give a characterization of Gorenstein projective functors in $\mmod \underline{\CX}$ (Theorem \ref{Classification of Gorenstein projective functors}). This    characterization is used  to define the functor ${}_{\underline{\CY}}\hat{\varUpsilon}_{\underline{\CX}}:\mmod \underline{\CY}\rt \rm{Gprj}\mbox{-}\underline{\CX}$, which is the second important functor in the paper.  In Section 5,  we will begin  with the definition of almost split sequences for the subcategories which are closed under extensions, and then we prove that the functors $\Psi_{\CX}$ and ${}_{\underline{\CY}}\hat{\varUpsilon}_{\underline{\CX}}$ are used as a tool to transfer the almost split sequences between the corresponding categories. In Section 6, the notion of the Auslander-Reiten quivers for a Krull-Schmidt category is recalled and then as a direct consequence of the results of Section 5, Theorem \ref{main4} is stated to give a comparison between the Auslander-Reiten quivers   $\Gamma_{\underline{\CY}}$, $\Gamma_{\rm{Gprj}\mbox{-} \underline{\CX}}$ and $\Gamma_{\CS_{\CY}(\La)}$. Finally,  some examples are provided to explain that how our results are helpful to compute the Auslander-Reiten quivers of the subcategories of Gorenstein projective modules.
\section{Preliminaries}
\subsection{Functor categories}\label{simple}
Let $\CA$ be an additive category  and $\mathcal{C}$ a subcategory of $\CA.$ We denote by $\Hom_{\CA}(X, Y)$ the set of morphisms from $X$ to $Y.$  Denote by $\rm{Ind}\mbox{-}\CA$ the set of isomorphisms classes of indecomposable objects in $\CA.$
An (right) $\CA$-module is a contravariant additive functor from $\CA$ to the category of abelian groups. We call an $\CA$-module $F$ \emph{finitely presented} if there exists an exact sequence $\Hom_{\CA}(-, X) \st{f} \rt \Hom_{\CA}(-, Y) \rt F \rt 0.$ 
We denote by $\mmod \CA$ the category of finitely presented $\CA$-modules. We call $\mathcal{C}$ contravariantly (resp. covariantly) finite in $\CA$ if $\Hom_{\CA}(-, X)|_{\mathcal{C}}$ (resp. $\Hom_{\CA}(-, X)|_{\mathcal{C}}$) is a finitely generated $\mathcal{C}$-module for any $X$ in $\CA.$ We call $\mathcal{C}$ \emph{functorially finite} if it is both contravariantly and covariantly finite.
 It is known that if $\mathcal{C}$ is a contravariantly finite subcategory of abelian category $\CA$, then $\mmod\mathcal{C}$ is an abelian category.

Let $\CA$ be an abelian category with enough projectivs and  $\mathcal{C}$ contains all projective objects of $\CA.$  We consider the stable category of $\mathcal{C}$, denoted by $\underline{\mathcal{C}}.$ The objects of $\underline{\mathcal{C}}$ are the same as the objects of $\mathcal{C}$, which we usually denote by  $\underline{X}$ when an object $X \in \mathcal{C}$ considered as an object in the stable category, and  the morphisms are given by $\Hom_{\underline{\mathcal{C}}}(\underline{X}, \underline{Y})= \Hom_{\mathcal{C}}(X, Y)/ \CP(X, Y)$, where $\CP(X, Y)$ is the subgroup of $\Hom_{\mathcal{C}}(X, Y)$  consisting of those morphisms from $X $ to $Y$ which factor through a projective object in $\CA.$ We also denote by $\underline{f}$ the residue class of $f: X \rt Y$ in $\Hom_{\underline{\mathcal{C}}}(\underline{X}, \underline{Y})$. In order to simplify, we will use $(-, X)$, resp. $(-, \underline{X})$, for the representable functor $\Hom_{\mathcal{C}}(-, X)$, resp. $\Hom_{\underline{\mathcal{C}}}(-, \underline{X})$, in $\mmod \mathcal{C}$, resp. $\mmod \underline{\mathcal{C}}.$ It is well-known that the canonical functor $\pi: \mathcal{C}\rt \underline{\mathcal{C}}$ induces a fully faithful functor  functor $\pi^*:\mmod \underline{\mathcal{C}}\rt \mmod \mathcal{C}$. Hence due to this embedding we can  identify the functors in $\mmod \underline{\mathcal{C}}$ as those functors in $\mmod \mathcal{C}$ which vanish on the projective objects.   

We  assume throughout this paper that $\La$ is an Artin algebra over a commutative artinian ring $k$. A subcategory $\CX$ of $\mmod \La$ is {\it always} a {\it full} subcategory of $\mmod \La$  {\it closed under isomorphisms, finite direct sums and direct summands.} The subcategory $\CX$ is called of finite representation type if $\Ind\CX$ is a finite set. Here, $\rm{Ind}\mbox{-}\CX$ denotes the set of all non-isomorphic indecomposable modules in $\CX$. An Artin algebra  $\La$ is called of finite representation type, or  simply representation-finite,  if $\mmod\La$ is of finite representation  type. If $\CX$ is of finite representation type, then it admits a representation generator, i.e., there exists $X \in \CX$ such that $\CX=\add X$, the subcategory of $\mmod \La$ consisting of all direct summands of all finite direct sums of copies of $X$. It is known that $\add X$ is a functorially finite subcategory of $\mmod\La$. To avoid complicated notation, for $\rm{prj}\mbox{-} \Lambda \subset \CX \subset \mmod \La$, we show $\Hom_{\CX}(X, Y)$, resp. $\Hom_{\underline{\CX}}(\underline{X}, \underline{Y}),$ by $\Hom_{\La}(X, Y)$, resp. $\underline{\Hom}_{\La}(X, Y)$. Set $\rm{Aus}(\CX, X)=\End_{\La}(X)$, when $\CX$ is a subcategory with representation generator $X$. Clearly $\rm{Aus}(\CX, X)$ is an Artin algebra. It is known that the evaluation functor $\zeta_X:\mmod\CX \lrt \mmod\rm{Aus}(\CX, X)$ defined by $\zeta_X(F)=F(X)$, for $F \in \mmod\CX$, gives an equivalence of categories, see \cite[Proposition 2.7 (c)]{Au1}. Let us give a proof of this fact. Set $A=\rm{Aus}(\CX, X)$. Density: assume an $A$-module $N$ is given. The contravariant functor $F: \CX\rt \rm{A}b$ (not necessarily additive) is defined as follows: First, let $F(X)=N$ and for any $f:X\rt X$, let $F(f)=r_f$ where $r_f:N\rt N$ is defined by using  right $A$-module  structure of $N$, i.e., $x \in N \mapsto xf.$  For any $n, m> 0$, let $F(X^n)=N^n$, and for  any morphism $f: X^n\rt X^m$, with matrix factorization $[f_{ji}]_{1\leqslant i \leqslant n}^{ 1 \leqslant j \leqslant m}$,  $F(f): N^n\rt N^m$ is defined as $[r_{f_{ji}}]_{1\leqslant i \leqslant n}^{ 1 \leqslant j \leqslant m}$ . Now, let $M$ be a direct summand of $X^n$ for some $n>0$. There exists an idempotent $a:X^n\rt X^n$ such that $\text{Im} \ a=M$. Define $F(M)=\text{Im} \ F(a)$. By construction $\zeta_X(F)=N$.  Faithfulness: it follows from this fact that the regular module  $\La$ is isomorphic to some direct summand of $X.$ Fullness: assume an $A$-homomorphism $g:F(X)\rt G(X)$ is given. The same way as the proof of the density works to construct a natural transformation $\sigma:F\rt G$ such that $\zeta_X(\sigma)=g.$ \\ The functor $\zeta_X$ also induces an equivalence of categories $\mmod\underline{\CX} \simeq \mmod\rm{Aus}(\underline{\CX}, \underline{X})$, only by the  restriction. Recall that $\rm{Aus}(\underline{\CX}, \underline{X})=\End_{\La}(X)/\CP$, where $\CP=\CP(X, X)$.  The Artin algebra $\rm{Aus}(\CX, X)$, resp. $\rm{Aus}(\underline{\CX}, \underline{X})$, is called the  relative, resp. stable, Auslander algebra of $\La$ with respect to the subcategory $\CX$ and 
 the representation generator $X$. For the case $\CX=\mmod \La$, we obtain the (stable) Auslander algebras.  If $X'$ is another representation generator of $\CX$, then $\rm{Aus}(\CX, X)$, resp. $\rm{Aus}(\underline{\CX}, \underline{X})$, and  $\rm{Aus}(\CX, X')$, resp. $\rm{Aus}(\underline{\CX}, \underline{X'}),$ are Morita equivalent.  If no ambiguity may rise, we drop $X$ and $\underline{X}$ in the relevant notations.

\subsection{Gorenstein projective objects}
Let $\CA$ be an abelian category with enough projectives. 
A complex $$P^\bullet:\cdots\rightarrow P^{-1}\xrightarrow{d^{-1}} P^0\xrightarrow{d^0}P^1\rightarrow \cdots$$ of  projective objects in $\CA$ is said to be \emph{totally acyclic} provided it is acyclic and the Hom complex $\Hom_{\La}(P^\bullet, Q)$ is also acyclic for any projective object $Q$ in $\CA$. An object $M$ in $\CA$ is said to be  \emph{Gorenstein projective} provided that there is a totally acyclic complex $P^\bullet$ of   projective objects over $\CA$ such that $M \simeq  \text{Ker}\ d^0$.  We denote by $\rm{Gprj}\mbox{-}\CA$ the full subcategory of $\CA$ consisting of all Gorenstein projective objects in $\CA$. When, for an additive category $\mathcal{C}$, $\mmod \mathcal{C}$ is abelian, we will use $\rm{Grpj}\mbox{-}\mathcal{C}$ to denote the subcategory of  Gorenstein projective objects in $\mmod \mathcal{C}$. Also, for simplicity, when $\CA=\mmod \La$, the subcategory of Gorenstein projective modules in $\mmod \La$ is denoted by $\rm{Gprj}\mbox{-}\La.$

An Artin algebra $\La$ is  $\rm{CM}$-finite, if there are only finitely many isomorphism classes of  finitely generated Gorenstein projective indecomposable $\La$-modules. Clearly, $\La$ is a $\rm{CM}$-finite algebra if and only if there is a finitely generated module $E$ such that $\rm{Gprj} \mbox{-} \La=\rm{add} \mbox{-}E$.  
 If $\La$ is self-injective, then $\rm{Gprj} \mbox{-} \La= \mmod \La,$ so $\La$ is $\rm{CM}$-finite if and only if $\La$ is representation-finite. If $E$ is a   representation generator of $\rm{Gprj} \mbox{-} \La$, then the relative (resp. stable) Auslander algebra  $\rm{Aus}(\rm{Gprj} \mbox{-} \La, E)=\rm{End}_{\La}(E)$, resp. $\rm{Aus}(\underline{\rm{Gprj}} \mbox{-} \La, \underline{E})=\underline{\rm{End}}_{\La}(E)$, in the literature is usually called the (resp. stable) Cohen-Macaulay Auslander algebra of $\La.$ The reason of such a naming might be since in some context used ``Cohen-Macaulay'' instead of ``Gorenstein-projective''. 
 \subsection{Triangular matrix algebras}
 Let $\CA$ be an abelian category and let $H(\CA)$ be the morphism category over  $\CA.$ Indeed, the objects in $H(\CA)$ are the morphisms in $\CA$, and morphisms are given by the commutative diagrams. We can consider  objects in $H(\CA)$ as the representations over the quiver $\mathbb{A}_2:v \st{a}\rt w$ by objects and morphisms in $\CA,$ usually denoted by $\rm{rep}(\mathbb{A}_2, \CA).$ In case that $\CA=\mmod \La$ we know by a  general fact the category $\rm{rep}(\mathbb{A}_2, \mmod \La)$, or simply $\rm{rep}(\mathbb{A}_2,  \La)$,  is equivalent to the category of finitely generated right module over the path algebra $\La\mathbb{A}_2 \simeq T_2(\La)$, where  $T_2(\La)= \tiny {\left[\begin{array}{ll} \La & 0\\ \La & \La \end{array} \right]}$, lower triangular $2\times 2$ matrix algebra over $\La.$ Since the categories $\rm{rep}(\mathbb{A}_2, \La)$ and  $H(\mmod \La)$, for short $H(\La)$, are equivalent to $\mmod T_2(\La)$, then by these equivalences we can naturally define the notion of  Gorenstein projective representations (or morphisms) in $\rm{rep}(\mathbb{A}_2, \La)$ (or in $H(\La)$), which comes from the concept of Gorenstein projective modules over $T_2(\La)$. There is the following local characterization of Gorenstein projective representations in $\rm{rep}(\mathbb{A}_2, \La)$:
 
 \begin{lemma}(\cite[Theorem 3.5.1 ]{EHS} or \cite[Theorem 5.1]{LZ1}) \label{Gorenproj charecte}
 	Let $(X \st{f}\rt Y)$ be a representation  in $\rm{rep}(\mathbb{A}_2, \La)$. Then $(X \st{f} \rt Y)$ is a Gorenstein projective representation  if and only if $(1)$ $X, Y$ and $\text{Cok} \ f$ are in $\rm{Gprj}\mbox{-}\La$,  and $(2)$ $f$ is a monomorphism. 	
 \end{lemma}
Throughout  the paper, we completely free use the identification between objects in $H(\La)$ and modules in $\mmod T_2(\La)$.
\section{An equivalence}\label{Equivalence}
let $\CA$ be an abelian category with enough projectives in this section. Following \cite{MT}, a subcategory $\mathcal{C}$ of $ \CA$ is called {\it quasi-resolving}
if it contains  the projective objects of $\CA$, closed under finite direct sums and closed under kernels of epimorphisms in $\mathcal{C}$. Moreover, a quasi-resolving subcategory $\mathcal{C}$ of $\CA$ is called {\it resolving} if it is closed under direct summands and  extensions. In this case, as mentioned in \cite[Proposition 2.11]{MT}, $\mmod \underline{\mathcal{C}}$ is an abelian category with enough projectives, although $\mmod \mathcal{C}$ might not  be always abelian. In this section we show that for a quasi-resolving subcategory $\mathcal{C}$ of $\CA$, the category $\mmod \underline{\mathcal{C}}$ is realized by the (additive) quotient category of a subcategory of the morphism category $H(\CA)$ modulo a relation generated by some objects.

Assume $\mathcal{C}$ is a full subcategory of an additive category $\mathcal{D}$. Denote by $[\mathcal{C}]$, the ideal of morphisms in $\mathcal{D}$ which factor through an object in $\mathcal{C}$. The quotient category $\mathcal{D}/\mathcal{C}$ has the same objects as $\mathcal{D}$ but $\rm{Hom}$-space of morphisms
$$\Hom_{\mathcal{D}/\mathcal{C}}(X, Y):=\Hom_{\mathcal{D}}(X,Y)/[\mathcal{C}](X,Y).$$

For a given subcategory $\mathcal{C}$ of an abelian category  $\CA$,  we assign the subcategory $\CS_{\mathcal{C}}(\CA)$ of $H(\CA)$ consisting of all  morphisms $(A \st{f} \rt B)$ satisfying: $(i)$ $f$ is a monomorphism; $(ii)$ $A$, $B$ and $\text{Cok} \ f$  belong to $\mathcal{C}.$

In the case that $\CA=\mmod \La$,  for $\CX \subseteq \mmod \La$ we show $\CS_{\CX}(\mmod \La)$  by $\mathcal{S}_{\CX}(\La).$

We define a functor $\Psi_{\mathcal{C}}:\CS_{\mathcal{C}}(\CA) \rt \mmod \underline{\mathcal{C}}$ with respect to a   subcategory $\mathcal{C}$ of $\CA$ as follows.

\begin{construction}\label{FirstCoonstr}

	Taking an object $(A \st{f} \rt B)$ of $\CS_{\mathcal{C}}(\CA)$, then we have the following short exact sequence
	$$0 \rt A \st{f} \rt  B \rt \text{Cok} \ f \rt 0$$ in $\CA, $ this in turn gives the following short exact sequence
	$$ (*) \ \ \ 0 \lrt (-, A) \st{(-, f )} \rt (-, B) \rt (-, \text{Cok} \ f) \rt F \rt 0$$
	in $\mmod \mathcal{C}$. In fact, $(*)$  corresponds to a projective resolution of $F$ in $\mmod \mathcal{C}.$ We define $\Psi_{\mathcal{C}}(A \st{f} \rt B):= F$. 
	
	For morphism: Let $\sigma=(\sigma_1, \sigma_2)$ be a morphism from $(A \st{f} \rt B)$ to $(A' \st{f'} \rt B')$, it gives the following commutative  diagram
	\[\xymatrix{0 \ar[r] & A \ar[r]^{f} \ar[d]^{\sigma_1} & B \ar[r] \ar[d]^{\sigma_2} & \text{Cok} \ f  \ar[r] \ar[d]^{\sigma_3} & 0 \\
		0 \ar[r] & A'  \ar[r]^{f'} & B' \ar[r] & \text{Cok} \ f'  \ar[r]  & 0.}\]
	 By applying the Yoneda functor, the above diagram gives  the following commutative  diagram
	
	\[\xymatrix{0 \ar[r] & (-, A) \ar[d]_{(-, \sigma_1)} \ar[r]^{(-, f)} & (-, B) \ar[d]^{(-, \sigma_2)} \ar[r] & (-, \text{Cok} \  f) \ar[d]^{(-, \sigma_3)} \ar[r] & F \ar[d]^{\overline{(-, \sigma_3)}} \ar[r] & 0 \\ 0  \ar[r] & (-, A')\ar[r]^{(-, f')} & (-, B')\ar[r] & (-,\text{Cok} \ f' )\ar[r] & F' \ar[r] & 0.}\]
	in $\mmod \mathcal{C}.$ Define $\Psi_{\mathcal{C}}(\sigma):= \overline{(-,\sigma_3)}$.
	
\end{construction}
Following \cite{RZ1},  an additive  functor $F:\CA  \rt \CB$ will be said to be objective provided any morphism $f: A \rt A'$ in $\CA $ with $F(f) = 0$ factors through an object $C$  that $F(C)=0.$ 

The following result  was first stated in the unpublished work \cite{H1} by the author.
\begin{theorem} \label{Thefirst}
	Let $\CA$ be an abelian category with enough projectives.
	Let $\mathcal{C}$ be a quasi-resolving  subcategory of $\CA$.  Consider the full subcategory $\mathcal{V}$ of $\CS_{\mathcal{C}}(\CA)$ formed by  all finite direct sums of   objects of the form  $(C \st{1}\rt C )$ and $ (0 \rt C)$, which $C$ runs through   of objects in $\mathcal{C}$. Then the functor $\Psi_{\mathcal{C}}$, defined in the above construction, induces the following equivalence of categories
	$$ \CS_{\mathcal{C}}(\CA)/ \mathcal{V} \simeq \mmod \underline{\mathcal{C}}.$$
	Moreover, if $\mathcal{C} \subseteq \mmod \La$ is of finite representation type, then the following conditions are equivalent.
	\begin{itemize}
		\item [$(1)$] The relative stable Auslander algebra $\rm{Aus}(\underline{\mathcal{C}})$ is of finite representation type;
		\item[$(2)$] The subcategory $\CS_{\mathcal{C}}(\La)$ of $\mmod T_2(\La)$ is of finite representation type.
	\end{itemize}
\end{theorem}
  
\begin{proof}
	The functor $\Psi_{\mathcal{C}}$ is dense.  Let $F \in \mmod \underline{\mathcal{C}}$. By identifying $F$ as an object in $\mmod \mathcal{C}$ we have a projective presentation $(-, B) \st{(-, g)} \rt (-, D) \rt F \rt 0$ such that $g$ is an epimorphism.  The assumption $\mathcal{C}$ being closed under kernels of epimorphisms implies $\text{Ker} \ g$ belongs to $\mathcal{C}.$ Hence we get the projective resolution $0 \rt (-, \text{Ker} \ g) \rt (-, B) \rt (-, D) \rt F \rt 0,$ in $\mmod \mathcal{C}$ and consequently $\Psi_{\mathcal{C}}(\text{Ker} \ g \rt B)=F.$ Any morphism between  two functors $F=\Psi_{\mathcal{C}}(A\st{f}\rt B)$ and $\Psi_{\mathcal{C}}(C\st{g}\rt D)$ in $\mmod \underline{\mathcal{C}}$ can be   considered  as a morphism in $\mmod \mathcal{C}$, so it  can be  lifted to the  corresponding projective resolutions of $F$ and $G$ in $\mmod \mathcal{C}$ induced by $(A\st{f}\rt B)$ and $(C\st{g}\rt D)$. Then by using Yoneda's Lemma we can obtain a morphism in $\CS_{\mathcal{C}}(\CA)$ to prove fullness. If $\Psi_{\mathcal{C}}(A \st{f} \rt B)=0$, then by the definition
	 we have the following exact sequence in $\mmod \mathcal{C}$
	$$0 \rt (-, A) \st{(-, f)} \lrt (-, B)  \st{(-, g)} \lrt (-, \text{Cok} \ f) \rt 0.$$
	Now by putting $\text{Cok} \ f$, as it belongs to $\mathcal{C},$ in  the above short exact sequence   we get
	$$0 \rt A \st{f} \rt B\st{g} \rt\text{Cok} \ f \rt 0$$
	is split. This gives us $(A \st{f} \rt B)$ being isomorphic to the  object $(A \rt \text{Im} \ f) \oplus (0 \rt \text{Ker} \ g)$ in $H(\CA)$, and so $f$ belongs to $\mathcal{V}$. Assume that $\Psi_{\mathcal{C}}(\sigma)=0$, for $\sigma=(\sigma_1, \sigma_2): (A \st{f} \rt B) \rt (A' \st{f'} \rt B')$ in $\CS_{\mathcal{C}}(\CA)$. Therefore, we have
	
	\[\xymatrix{0 \ar[r] & (-, A) \ar[d]_{(-, \sigma_1)} \ar[r]^{(-, f)} & (-, B) \ar[d]^{(-, \sigma_2)} \ar[r] & (-, \text{Cok} \ f) \ar[d]^{(-, \sigma_3)} \ar[r] & F \ar[d]^{0} \ar[r] & 0 \\ 0  \ar[r] & (-, A')\ar[r]^{(-, f')} & (-, B')\ar[r] & (-,\text{Cok} \ f' )\ar[r] & F' \ar[r] & 0.}\]
	
	Since the first (resp. second) row of the above diagram is a projective resolution for $F$ (resp. $F'$) in $\mmod \mathcal{C}$. Then by considering the following commutative diagram
	\[\xymatrix{\cdots \ar[r] & (-, A) \ar[d]_{(-, \sigma_1)} \ar[r]^{(-, f)} & (-, B) \ar[d]^{(-, \sigma_2)} \ar[r] & (-, \text{Cok} \ f) \ar[d]^{(-, \sigma_3)} \ar[r] &  \cdots \\ \cdots  \ar[r] & (-, A')\ar[r]^{(-, f')} & (-, B')\ar[r] & (-,\text{Cok} \ f' )\ar[r] &  \cdots,}\]
	as a chain map in $\mathbb{C}^{\rm{b}}(\mmod \mathcal{C})$, the category of bounded complexes on $\mmod \mathcal{C},$ the chain map should be null-homotopic. Hence, the above chain map should  be factored through a contractible complex in $\mathbb{C}^{\rm{b}}(\mmod \mathcal{C})$ as follows:

	$$\xymatrix{    & 0 \ar[r] & (-, A)   \ar@/^1.25pc/@{.>}[dd]^<<<<<{(-, \sigma_1)} \ar[d]  \ar[r]  & (-, B)  \ar@/^1pc/@{.>}[dd]^<<<<<<<<{(-, \sigma_2)} \ar[d]  \ar[r]  &  (-, \text{Cok} \ f)  \ar@/^1.2pc/@{.>}[dd]^<<<<<{(-, \sigma_3)} \ar[d]  \ar[r]  & 0 \ar@/^1.2pc/@{.>}[dd] \ar[d] &   \\
		&0 \ar[r] & (-, A') \ar[r] \ar[d] & (-, A'\oplus B') \ar[r] \ar[d]  &   (-, B'\oplus\text{Cok} \ f' )  \ar[d]  \ar[r]  & (-, \text{Cok}\ f')
		 \ar[d]  \\
		& 0 \ar[r] & (-, A')  \ar[r]  & (-, B')  \ar[r]  &(-,\text{Cok} \ f') \ar[r] & 0. }$$
	
	This factorization as the above  gives us a factorization of morphism $\sigma$ through the direct sum of $(A' \st{1} \rt A')$  and $(0 \rt B')$ in  $\CS_{\mathcal{C}}(\CA).$ So $\Psi_{\mathcal{C}}$ is an objective functor. 	Summing up,  $\Psi_{\mathcal{C}} $ is  full, dense and objective. Hence by \cite[Appendix]{RZ1}, $\Psi_{\mathcal{C}} $ induces an equivalence between $\CS_{\mathcal{C}}(\CA)/ \mathcal{V}$ and $\mmod \underline{\mathcal{C}}$. So we are done.\\
	Assume $\mathcal{C}$ is a subcategory of finite representation type of $\mmod \La.$	By the first part we have the equivalence $\CS_{\mathcal{C}}(\La)/ \CV\simeq \mmod \underline{\mathcal{C}}.$ Due to the equivalence we observe there is a bijection between indecomposable functors  in $\mmod \underline{\mathcal{C}}$ and the indecomposable objects in $\CS_{\mathcal{C}}(\La)$ except a finite number. Hence, we can deduce that $\mmod \underline{\mathcal{C}}$ is of finite representation type if and only if so is $\CS_{\mathcal{C}}(\La)$. We know by subsection \ref{simple}, $\mmod \underline{\mathcal{C}}\simeq \mmod  \rm{Aus}(\underline{\mathcal{C}})$. The latter equivalence completes the second part of the statement. 
\end{proof}
In fact, the above theorem is a relative version of  the equivalences given in \cite{RZ1} and \cite{E}.
\begin{remark}
	As we were informed later by an anonymous  referee the  equivalence given in the above theorem is a  consequence of \cite[Proposition 3.8 and Corollary 3.9 (i)]{B} in a completely different approach. But what is important more for us  here not really the equivalence itself, as will be shown in the next sections, to  see that how the functor $\Psi_{\mathcal{C}}$  helps to transfer the    representation theory  from  $\CS_{\mathcal{C}}(\La)$  to $\mmod \underline{\mathcal{C}}$.
\end{remark}

 By using  Lemma \ref{Gorenproj charecte}, we observe $\CS_{\rm{Gprj}\mbox{-}\La}(\La)\simeq \rm{Gprj}\mbox{-}T_2(\La)$. Thus by Theorem \ref{Thefirst}, we can say that for a $\rm{CM}$-finite algebra $\La$: $T_2(\La)$ is $\rm{CM}$-finite if and only if the stable Cohen-Macaulay Auslander algebra of $\La$  is representation-finite. In particular, if assume $\La$ is a self-injective of finite representation
   type, then $T_2(\La)$ is $\rm{CM}$-finite if and only if the stable Auslander  algebra of $\La$ is representation-finite.   
   
   Let us give an easy application by the above observation in the following example.
   \begin{example}\label{Example 3.5} Let $A=k\CQ/I$ be a quadratic monomial algebra, i.e. the ideal $I$ is generated by paths of length two. By \cite[Theorem 5.7]{CSZ}, $\underline{\rm{Gprj} }\mbox{-}A \simeq \CT_1 \times \cdots \times \CT_n $ such that the underlying categories of triangulated categories  $\CT_i$ are equivalent to semisimple abelian categories $\mmod k^{d_i}$ for some natural numbers $d_i.$ Hence $ \mmod \underline{\rm{Gprj} }\mbox{-}A $ is a semisimple abelian category and consequently $\rm{Aus}(\underline{\rm{Gprj}}\mbox{-}A)$ is  a semisimple Artin algebra. Thus Theorem \ref{simple}
   	implies $T_2(A)$ is $\rm{CM}$-finite.    	
   \end{example}

\section {Syzygies}\label{Section 4}
Throughout this section let $\CX \subseteq \mmod \La$ be a  contravariantly finite and quasi-resolving subcategory of $\mmod \La.$ \\
We begin this section by giving an explicit description of the syzygies (up to projective summands) of any functor $F$ in $\mmod \underline{\CX}$ via their projective resolutions in $\mmod \CX$. This description helps us to give a classification of the non-projective Gorenstein projective functors in $\mmod \underline{\CX}$ via their minimal  projective resolutions in $\mmod \CX.$ To have such a classification we need some more conditions on $\CX$. By such a classification, as our main aim of this section,  we will define a fully faithful exact (and almost dense, see Definition \ref{almsotdense}, when $\CX$ is of finite representation type) functor from $\mmod \underline{\CY}$ to $\rm{Gprj}\mbox{-}\underline{\CX}$, where $\CY=\CX\cap\rm{Gprj}\mbox{-}\La$.

Given a $\La$-module $M, $ denote the kernel of the projective cover $P_M \rt M$ by $\Omega_{\La}(M)$. The module  $\Omega_{\La}(M)$ is called the first syzygy of $M$. We let $\Omega^0_{\La}(M)=M$ and then inductively for each $i \geq 1$, set $\Omega^i_{\La}(M)=\Omega_{\La}(\Omega^{i-1}_{\La}(M)).$ Similarly, when $\mmod \mathcal{C}$ has projective covers, one can define the $n$-the syzygy of a functor $F$ in $\mmod \mathcal{C}$, and denoted by $\Omega^n_{\mathcal{C}}(F)$. The categories $\CX$ and $\underline{\CX}$ both are  varieties of annul, in the scene of \cite{Au1}. On the other hand, since for each $X$, resp. $\underline{X}$, in $\CX$, resp. $\underline{\CX}$, $\rm{End}_{\La}(X)$, resp. $\underline{\rm{End}}_{\La}(\underline{X})$,  is clearly Artin algebra, then by \cite[Corollary 4.13]{Au1}, $\mmod \CX$ and $\mmod \underline{\CX}$ both have projective covers. In view of \cite[Corollary 4.4]{K}, they are Krull-Schmidt as well. We recall that an  additive category $\mathcal{C}$ is called Krull-Schmidt  category if every object decomposes into a finite direct sum of objects having local endomorphism rings. 

For $n \geqslant 0$,  let $\Omega^n(\CX)$ denote the subcategory of $\mmod \La$ consisting of all modules $M$ such that $M\simeq Q \oplus N$, where $Q \in \rm{prj}\mbox{-} \La$ and $N=\Omega^n_{\La}(X)$ for some $X$ in $\CX.$\\

 \begin{proposition}\label{Thefirst syzygy}
  Let $F \in \mmod \underline{\CX}$ and $0 \rt (-, A)\rt (-, B)\rt (-, C)\rt F\rt 0$ be a  projective resolution of $F$ in $\mmod \CX.$ Then,  there is a short exact sequence  $0 \rt G\rt  (-, \underline{C})\rt F \rt 0$, in $\mmod \underline{\CX}$, such that $G$ has the following projective resolution 
 	$0 \rt (-, \Omega_{\La}(C))\rt (-, A\oplus P_C)\rt (-, B)\rt G \rt0$ 	
 	in $\mmod \CX,$ where $P_C$ is the projective cover of $C$ in $\mmod \La.$ 
 \end{proposition}
\begin{proof} 
	Since $F$ vanishes on projective modules,  we have the  exact sequence $(-, \underline{B})\rt (-, \underline{C}) \rt F \rt 0$ in $\mmod \underline{\CX}$. Letting $G$ be the kernel of  the epimorphism $(-, \underline{C})\rt F \rt 0$, we have the short exact sequence $0 \rt G \rt (-, \underline{C})\rt F \rt 0$ in $\mmod \underline{\CX}$.
Consider the following pull-back diagram in $\mmod \CX$:
$$\xymatrix{& & 0 \ar[d] & 0 \ar[d]& &\\
	& & K_1 \ar@{=}[r] \ar[d] & K_1 \ar[d]& &\\
	0 \ar[r] & G\ar@{=}[d] \ar[r] & M
	\ar[d] \ar[r] &(-, C) \ar[d] \ar[r] & 0\\
	0 \ar[r] & G \ar[r] & (-, \underline{C}) \ar[r] \ar[d] & F \ar[d] \ar[r] & 0 &\\
	& & 0 & 0 & & }
$$ 
where the rightmost column obtained from the projective resolution in the assertion.

By using the horseshoe lemma, we get the following commutative diagram with exact columns and rows and the middle row splitting:
$$\xymatrix{& 0 \ar[d] & 0 \ar@{-->}[d] & 0 \ar[d]& &\\
	0 \ar@{-->}[r] & (-, A) \ar[d] \ar@{-->}[r] & D_1 \ar@{-->}[d]
	\ar@{-->}[r] & H_1 \ar[d] \ar@{-->}[r] & 0\\
	0 \ar[r] & (-, B) \ar[d] \ar[r] & (-, B)\bigoplus (-, C)
	\ar@{-->}[d] \ar[r] & (-, C) \ar[d] \ar[r] & 0\\
	0 \ar[r] & K_1 \ar[d] \ar[r] & M
	\ar@{-->}[d] \ar[r] & (-, \underline{C}) \ar[d] \ar[r] & 0\\
	& 0  & 0  & 0 & }
$$

Note that in the above digram $H_1$ is obtained by the  projective resolution $0 \rt (-, \Omega_{\La}(C))\rt (-, P_C)\rt (-, C)\rt (-, \underline{C})\rt 0$ in $\mmod \CX$, and the leftmost short exact sequence induced by the projective resolution of $F$ in the assertion.
Also, we have the following pull-back diagram in $\mmod \CX$:
$$\xymatrix{& 0 \ar[d] & 0 \ar[d] & &\\
	&  D_1 \ar@{=}[r] \ar[d]& D_1 \ar[d] & & \\
	0 \ar[r] & (-, B) \ar[r] \ar[d] & (-, B)\bigoplus (-, C) \ar[r] \ar[d] & (-, C) \ar[r] \ar@{=}[d] & 0\\
	0 \ar[r] & G \ar[r] \ar[d] & M \ar[r] \ar[d] & (-, C)\ar[r] & 0\\
	& 0 & 0 & &}$$
Finally, we have the following pull-back diagram in $\mmod \CX$:
$$\xymatrix{&   & 0 \ar[d] & 0 \ar[d]& &\\
 &   &  (-, \Omega_{\La}(C)) \ar[d]
	\ar@{=}[r] & (-, \Omega_{\La}(C)) \ar[d] \ar[r] & 0\\
	0 \ar[r] &(-, A) \ar@{=}[d] \ar[r] &(-, A)\bigoplus(-, P_C)
	\ar[d] \ar[r] & (-, P_C) \ar[d] \ar[r] & 0\\
	0 \ar[r] & (-, A) \ar[d] \ar[r] & D_1
	\ar[d] \ar[r] & H_1 \ar[d] \ar[r] & 0\\
	& 0  & 0  & 0 & }
$$
By gluing the obtained short exact sequences $0 \rt D_1 \rt (-, B)\rt G \rt 0$ and $0 \rt (-, \Omega_{\La}(C))\rt (-, A\oplus P_C)\rt D_1 \rt 0$ from the last two above diagrams, we get the desired projective resolution in the statement.
\end{proof}

By iterating use of the above proposition we get the following corollary.

\begin{corollary}\label{N-the syzygy}
	 Let $F \in \mmod \underline{\CX}$ and $0 \rt (-, A)\rt (-, B)\rt (-, C)\rt F\rt 0$ be a  projective resolution of $F$ in $\mmod \CX.$ Then,  for each  $n >0$, we have the following situations.
		\begin{itemize}
			\item [$(1)$] If $n=3k$, then $\Omega^n_{\underline{\CX}}(F)\simeq G$ in the stable category $\underline{\rm{mod}}\mbox{-}\underline{\CX}$, where $G$ is settled in the following short exact sequence
			$$0 \rt (-, \Omega^k_{\La}(A))\rt (-, \Omega^k_{\La}(B)\oplus Q )\rt (-, \Omega_{\La}^k(C)\oplus P)\rt G \rt 0,$$ for some projective modules $P$ and $Q$ in $\prj \La,$
			 in $\mmod \CX.$
		\item[$(2)$]  If $n=3k+1$, then $\Omega^n_{\underline{\CX}}(F)\simeq G$ in the stable category $\underline{\rm{mod}}\mbox{-}\underline{\CX}$, where $G$ is settled in the following short exact sequence
		$$0 \rt (-, \Omega^{k+1}_{\La}(C))\rt (-, \Omega^k_{\La}(A)\oplus Q )\rt (-, \Omega_{\La}^k(B)\oplus P)\rt G \rt 0,$$ for some projective modules $P$ and $Q$ in $\prj \La,$
		in $\mmod \CX.$
		\item [$(3)$]If $n=3k+2$, then $\Omega^n_{\underline{\CX}}(F)\simeq G$ in the stable category $\underline{\rm{mod}}\mbox{-}\underline{\CX}$, where $G$ is settled in the following short exact sequence
		$$0 \rt (-, \Omega^{k+1}_{\La}(B))\rt (-, \Omega^{k+1}_{\La}(C)\oplus Q )\rt (-, \Omega_{\La}^k(A)\oplus P)\rt G \rt 0,$$ for some projective modules $P$ and $Q$ in $\prj \La,$
		in $\mmod \CX.$  	
		\end{itemize}
\end{corollary}

\subsection{Gorenstein projective functors}
 As an application, we shall give a characterization of non-projective  Gorenstein projective indecomposable functors in $\mmod \underline{\CX}$ via their projective resolution in $\mmod \CX.$
\begin{proposition}\label{Gorenstein projective object in the stable catgeory via projective resolution}
	  Assume that there is $n\geqslant 0$ such that $\Omega^n(\CX)$ is contained in $\rm{Gprj}\mbox{-}\La$. If $F $ is a  non-projective Gorenstein projective indecomposable object in $\mmod \underline{\CX}$, then there is a projective resolution of $F$ in the following form
in $\mmod \CX$	$$0 \rt (-, A)\rt (-, B)\rt (-, C)\rt F\rt 0$$ such that 
	all the modules $A, B$ and $C$ are Gorenstein projective $\La$-modules.
\end{proposition}
\begin{proof}
	 By definition,  $F$ is  isomorphic to $3n$-syzygy of some functor $G$ in $\mmod \underline{\CX}$  in the stable category $\underline{\rm{mod}}\mbox{-}\underline{\CX}$. Consider   a projective resolution of $G$ in $\mmod \CX$ as the  following
	$$0 \rt (-, M) \rt (-, N)\rt (-, L)\rt G\rt 0.$$
By Corollary \ref{N-the syzygy}, the $3n$-th syzygy $\Omega^{3n}_{\underline{\CX}}(G)\simeq G'$ in $\underline{\rm{mod}}\mbox{-}\underline{\CX}$, such that $G'$ has the following projective resolution 
	$$0 \rt (-, \Omega^n_{\La}(M))\rt (-, \Omega^n_{\La}(N)\oplus Q )\rt (-, \Omega_{\La}^n(L)\oplus P)\rt G' \rt 0$$
in $\mmod \CX,$	for some projective modules $P$ and $Q$ in $\rm{prj}\mbox{-}\La.$ By our assumption $\Omega^n_{\La}(M), \Omega^n_{\La}(N)$ and $\Omega^n_{\La}(L)$ lie in $\rm{Gprj}\mbox{-}\La$. We observe that $G'\simeq F$ in $\underline{\rm{mod}}\mbox{-}\underline{\CX}$. So there are $X$ and $Y$ in $\CX$ such that $ G'\oplus (-, \underline{X})\simeq F \oplus (-, \underline{Y})$ in $\mmod\underline{\CX}$.  But since $\mmod \underline{\CX}$ is a Krull-Schmidt category and $F$ is a  non-projective indecomposable object, $F$ has to be isomorphic to  a direct summand of $G'$. As we have seen in the above $G'$ has a projective resolution in $\mmod \CX$ such that whose  terms are presented by the Gorenstein projective modules, clearly any direct summand of $G'$ is so. Thus $F$ has the desired projective resolution.
\end{proof}

 In the sequel, we shall investigate the converse of Proposition \ref{Gorenstein projective object in the stable catgeory via projective resolution}. We need the following lemma for our investigation.

Let $\mathcal{C}$ be an additive category such that $\mmod \mathcal{C}$ is abelian. For any $F$ and $G$ in $\mmod \mathcal{C}$ and $i>0$, we mean by $\Ext^i_{\mathcal{C}}(F, G)$, the $i$-th extension group  of $F$ by $G$.

\begin{lemma}\label{VAnishing Gorenstein left approximation}
	  If $F $ in $\mmod \underline{\CX}$ has a projective resolution in $\mmod \CX$ as the following 
	$$0 \rt (-, A)\rt (-, B)\rt (-, C)\rt F \rt 0$$
	with Gorenstein projective modules $A, B, C$ in $\mmod \La$, then $\Ext^i_{\underline{\CX}}(F, (-, \underline{X}))=0$ for any $X \in \CX$ and $i>0 $.
\end{lemma}
\begin{proof}
	Assume $F$ is a functor in $\mmod \underline{\CX}$ having a projective resolution as in the statement. Since $\mmod \underline{\CX}$ is a subcategory of $\mmod \CX$ closed under extensions, hence $\Ext^1_{\underline{\CX}}(F, (-, \underline{X}))\simeq \Ext^1_{\CX}(F, (-, \underline{X}))$. The latter group is the homology of the middle term in the following sequence
$$\Hom_{\CA}((-, C), (-, \underline{X}))\rt \Hom_{\CA}((-, B), (-, \underline{X}))\rt \Hom_{\CA}((-, A), (-, \underline{X})), $$
where $\CA=\mmod \CX$.	 Now by using the Yoneda lemma,  the above sequence is isomorphic to  the following sequence of abelian groups
$$\underline{\rm{Hom}}_{\La}(C, X)\rt\underline{\rm{Hom}}_{\La}(B, X)\rt \underline{\rm{Hom}}_{\La}(A, X).$$
So $\Ext^1_{\CX}(F, (-, \underline{X}))$ is isomorphic to the homology of the middle term of the above sequence. But this sequence is exact, see \cite[Lemma 2.2]{MT},  and consequently $\Ext^1_{\CX}(F, (-, \underline{X}))=0$.  We know $A, B$ and $C$ are in $\CX\cap\rm{Gprj}\mbox{-}\La.$ The condition  of being closed under kernels of epimorphisms for the subcategories $\CX$ and $\rm{Gprj}\mbox{-}\La$ implies that for any $i>0$, $\Omega^i_{\La}(A), \Omega^i_{\La}(B)$ and $\Omega^i_{\La}(C)$	lie in $\CX\cap\rm{Gprj}\mbox{-}\La.$ Thanks to Corollary \ref{N-the syzygy},  the $i$-th syzygies $\Omega^i_{\underline{\CX}}(F)$ have the same as $F$ a projective resolution in $\mmod \CX$ induced by Gorenstein projective modules.  The first part of the proof follows that $\Ext^1_{\underline{\CX}}(\Omega^i_{\underline{\CX}}(F), (-, \underline{X}))=0$ for any $X \in \CX$ and $i>0 $. Now by using the shifting dimension, one can conclude the proof. 
\end{proof}
Let us recall some notations which are useful for the proof of the next result. The {\it Gorenstein projective dimension} of an object $M$ in an  abelian category $\CA$, denoted by $\rm{Gpd} \ M$, is defined as the minimum of  the integers $n \geqslant 0$ such that there exists an exact sequence 
$$0 \rt G_n \rt G_{n-1}\rt \cdots \rt G_1\rt G_0 \rt M \rt 0,$$
in $\CA$ with $G_i$ Gorenstein projective objects. Let $\mathcal{C}$ be an additive category such that $\mmod \mathcal{C}$ is abelian. We say  that  $\mathcal{C}$ is {\it Gorenstein} if every functor  of $\mmod \mathcal{C}$ has  finite Gorenstein projective dimension. A  Gorenstein category $\mathcal{C}$ is called of {\it  Gorenstein projective dimension at most $n$} if every object of $\mmod \mathcal{C}$ has  Gorenstein projective dimension at most $n$. An Artin algebra $\La$ is said to be $n$-Gorenstein, if $\rm{prj}\mbox{-}\La$ is a Gorenstein category of Gorenstein projective dimension at most $n$.\\
Let $M$ be a module in $\mmod \La$, and  $n$  a positive integer.  Let $g:M\to Q_M$ be a minimal left ($\prj \La$)-approximation. 
Then the cokernel of $g$ is called the {\em first projective cosyzygy} of $M$ and denoted by $\Omega^{-1}_{\CP}(M)$.
The {\em $n$-th cosyzygy} $\Omega^{-n}_{\CP}(M)$ is defined inductively as $\Omega^{-1}_{\CP}(\Omega^{-(n-1)}_{\CP}(M))$.
We say that $\CX$ is {\it closed  under projective cosyzygies}, if for any $X \in \CX$, $\Omega^{-1}_{\CP}(X)$ lies in $\CX$.

Now we are ready to give the characterization of   non-projective Gorenstein projective indecomposable objects in $\mmod \underline{\CX}$ via their minimal projective resolutions in $\mmod \CX.$
\begin{theorem}\label{Classification of Gorenstein projective functors}
Let $\CX \subseteq \mmod \La$ be  closed under  projective cosyzygies.  Assume that there is $n\geqslant 0$ such that $\Omega^n(\CX)$ is contained in $\rm{Gprj}\mbox{-}\La$.	Let $F$ be a non-projective indecomposable   functor in $\mmod \underline{\CX}$ with the following minimal projective resolution
$$0 \rt (-, A)\rt (-, B)\rt (-, C)\rt F \rt 0.$$ Then $F$ is a Gorenstein projective object in $\mmod \underline{\CX}$ if and only if $A, B$ and $C$ are Gorenstein projective modules in $\mmod \La.$
\end{theorem}
\begin{proof}
The ``only if'' part follows from Proposition \ref{Gorenstein projective object in the stable catgeory via projective resolution}. For the inverse implication, assume $A, B$ and $C$ are Gorenstein projective. Due to \cite[Theorem 3.11]{MT}, the additive category $\underline{\CX}$ is Gorenstein of dimension at most $3n$. Hence the Gorenstein projective dimension of $F$ in $\mmod \underline{\CX}$ is finite. As the modules category, see \cite[Theorem 2.10]{Ho}, there exists a short exact sequence in $\mmod \underline{\CX}$
$$0 \rt L \rt G \rt F\rt 0,$$
with $G$ is a Gorenstein projective object  in $\mmod \underline{\CX}$ and the projective dimension of $L$ is finite in $\mmod \underline{\CX}.$ From Lemma \ref{VAnishing Gorenstein left approximation}, we observe that $\Ext^i_{\underline{\CX}}(F, (-, \underline{X}))=0$ for any $X \in \CX$ and $i>0 $. Therefore, $\Ext^1_{\underline{\CX}}(F, L)=0$. Hence the above short exact sequence must be split, so the claim follows.
\end{proof}
The above characterization might not be true for all the Gorenstein projective functors. In fact, consider an indecomposable module  $X \in \CX$ which does not belong to $\rm{Gprj}\mbox{-}\La$,  we have the following minimal projective resolution of $(-, \underline{X})$, 
$0 \rt (-, \Omega_{\La}(X))\rt (-, P_X)\rt (-, X)\rt (-, \underline{X})\rt 0 $
in $\mmod \CX.$ But $X$ does not lie in $\rm{Gprj}\mbox{-}\La.$

 We point out that  for a quasi-resolving subcategory  $\CX$ the condition `` $\Omega^n(\CX)$ is contained in $\rm{Gprj}\mbox{-}\La$ and closed under cosyzygies''  is called $(\bf{G}_n)$ in \cite{MT}.
 
\subsection{The second functor}
We begin by the following general construction which is essential to define another functor in similar to Construction \ref{FirstCoonstr}  to make a connection between a subcategory and  the modules category over an Artin algebra.
 \begin{construction}\label{Secounddconstruction}
 	Let $\mathcal{C}'$ be a full subcategory of an additive category $\mathcal{C}$.  As discussed in \cite[Section 3]{Au1}, the restriction functor $\rm{res}:\Mod \mathcal{C}\rt \Mod \mathcal{C}'$ admits left and right adjoints. Here $\Mod \mathcal{C}$  and $\Mod \mathcal{C}'$ denote the category of all contravariant functors from $\mathcal{C}$ and $\mathcal{C}'$ to the category of abelian groups, respectively.  Since the left adjoint ${}_{\mathcal{C}'}\varUpsilon_{\mathcal{C}}$  plays an important role in order to define our second promised functor in the paper, we devote this construction to describe the left adjoint in detail on the finitely presented functors. Let $F$ be a finitely presented functor
 	over $\mathcal{C}'$.  So there is an exact sequence $(-, C_2)\rt (-, C_1)\rt F \rt 0$ with $C_i$ in $\mathcal{C'}$, $i=1, 2$. We can naturally construct the functor $\tilde{F}$ in $\mmod \mathcal{C}$, by defining $\tilde{F}(C):=\text{Cok}(\Hom_{\mathcal{C}}(C, C_2)\rt \Hom_{\mathcal{C}}(C, C_1))$ for any $C$ in $\mathcal{C}$. Hence, by definition, we have the projective presentation $(-,C_2)\rt(-, C_1)\rt \tilde{F}\rt 0$ in $\mmod \mathcal{C}$. In the latter case, the functors $(-, C_i)$ are considered as representable functors in $\mmod \mathcal{C}$. For simplicity, the both cases are denoted by the same notation. By \cite[Proposition 3.1]{Au1}, ${}_{\mathcal{C}'}\varUpsilon_{\mathcal{C}}(F)=\tilde{F}.$ Let $F\st{\sigma}\rt G$ be a morphism in $\mmod \mathcal{C}'$. The morphism $\sigma$ can be lifted as the following to their existing projective presentations
 		$$\xymatrix{
 		(-, C_2) \ar[d]^{(-, f_2)} \ar[r] & (-, C_1) \ar[d]^{(-, f_1)}
 		\ar[r] &F \ar[d]^{\sigma}\ar[r] & 0&\\
 		 (-, C'_2)  \ar[r] & (-, C'_1)
 		\ar[r] & G\ar[r] &0.& \\ 	} 	 $$
 	in $\mmod \mathcal{C}'.$ For each $C \in \mathcal{C}$, $\tilde{\sigma}_C: \tilde{F}(C) \rt \tilde{G}(C)$ is defined as the following 
 		$$\xymatrix{
 		\Hom_{\mathcal{C}}(C, C_2) \ar[d]^{(C, f_2)} \ar[r] & \Hom_{\mathcal{C}}(C, C_1) \ar[d]^{(C, f_1)}
 		\ar[r] &\tilde{F }(C) \ar[d]^{\tilde{\sigma}_C}\ar[r] & 0&\\
 		\Hom_{\mathcal{C}}(C, C'_2)  \ar[r] & \Hom_{\mathcal{C}}(C, C'_1)
 		\ar[r] & \tilde{G}(C)\ar[r] &0.& \\ 	} 	 $$
 	Again by \cite[Proposition 3.1]{Au1}, ${}_{\mathcal{C}'}\varUpsilon_{\mathcal{C}}(\sigma)=\tilde{\sigma}.$ By the result of Auslander's paper, the functor ${}_{\mathcal{C}'}\varUpsilon_{\mathcal{C}}$ is  fully faithful.  We call this functor the extension functor from $\mathcal{C}'$ into $\mathcal{C}$. 	  
 \end{construction}

\begin{setup}\label{4.17}
	Let $\CX \subseteq \mmod \La$ be a contravariantly finite and quasi-resolving subcategory being  closed under  projective cosyzygies. Also assume that there is $n\geqslant 0$ such that $\Omega^n(\CX)$ is contained in $\rm{Gprj}\mbox{-}\La$.  Denote $\CY=\CX\cap \rm{Gprj}\mbox{-}\La$. It is clearly that $\CY$ is again 	 quasi-resolving subcategory of $\mmod \La.$ Note that by \cite[Theorem 1.4]{MT}, the equality $\Omega^n(\CX)=\CX\cap \rm{Gprj}\mbox{-}\La$ holds.
\end{setup}
Assume $\CX$ and $\CY$ are the same as in Setup \ref{4.17}. Specializing Construction \ref{Secounddconstruction} for $\underline{\CY}\subseteq \underline{\CX}$, we reach the extension functor ${}_{\underline{\CY}}\varUpsilon_{\underline{\CX}}:\mmod \underline{\CY}\rt \mmod \underline{\CX}$.  Let us remark that by identifying $\mmod \underline{\CY}$ and $\mmod \underline{\CX}$, respectively,  as subcategories of $\mmod \CY$ and $\mmod \CX$, we can consider ${}_{\underline{\CY}}\varUpsilon_{\underline{\CX}}$ as a restriction of the extension functor ${}_{\CY}\varUpsilon_{\CX}:\mmod\CY\rt \mmod \CX$.

 \begin{proposition}\label{Extextensionfunctor}
 Let $\CX$ and $\CY$ be as in Setup \ref{4.17}. Then the extension functor ${}_{\underline{\CY}}\varUpsilon_{\underline{\CX}}$, defined in the above,  is an exact functor. Moreover, its essential image is contained in $\rm{Gprj}\mbox{-}\underline{\CX}$.
 \end{proposition}
 \begin{proof}
 	To prove the latter claim in the statement, let $F \in \mmod\underline{\CY}$. Then there is an exact sequence in $\mmod \CY$
 	$$0 \rt (-, A)\rt (-, B)\rt (-, C)\rt F \rt 0$$
 	with  modules $A, B$ and $C$ in $\CY.$ The image $\tilde{F}$ under the functor ${}_{\underline{\CY}}\varUpsilon_{\underline{\CX}}$ has the following projective resolution 
 	$$0 \rt (-, A)\rt (-, B)\rt (-, C)\rt \tilde{F} \rt 0$$
 	in $\mmod \CX.$ Let us emphasis that here the representable  functors in the first exact sequence and the second exact sequence in the above,  considered as representable functors in $\mmod \CY$ and $\mmod \CX$, respectively. Now by the characterization given in Theorem \ref{Classification of Gorenstein projective functors}, we infer that $\tilde{F}$  lies in $\rm{Gprj}\mbox{-}\underline{\CX}$, as required. To prove the first part of the statement, we should show that the image of a short exact sequence in $\mmod \underline{\CY}$ is mapped to a short exact sequence in $\mmod \underline{\CX}.$ Take a short exact sequence $\eta:0 \rt F_1\rt F_2 \rt F_3\rt 0$ in $\mmod \underline{\CY}.$ 
 	Assume $0 \rt (-, X)\rt (-, Y)\rt (-, Z)\rt F_1\rt 0$ and $0 \rt (-, X')\rt (-, Y')\rt (-, Z')\rt F_3\rt 0$, respectively,  are projective resolutions of $F_1$ and $F_3$ in $\mmod \CY.$ By a standard argument,  we have the following  commutative diagram in $\mmod \CY$ with exact rows and the first three columns (from the left side hand) splitting:
 	$$\xymatrix{& 0 \ar[d] & 0 \ar[d] & 0 \ar[d]& 0\ar[d]& &\\
 		0 \ar[r] & (-, X) \ar[d] \ar[r] & (-, Y) \ar[d]
 		\ar[r] &(-, Z) \ar[d] \ar[r] & F_1\ar[r]\ar[d]&0\\
 		0 \ar[r] & (-, X\oplus X')  \ar[r]\ar[d] & (-, Y\oplus Y')\ar[d]
 		\ar[r] & (-, Z\oplus Z') \ar[r]\ar[d] & F_2\ar[r]\ar[d]& 0\\0\ar[r] &(-, X')\ar[r]\ar[d]&(-,  Y')\ar[r]\ar[d]&(-, Z')\ar[d] \ar[r]&F_3 \ar[r]\ar[d]&0 &\\ &0&0&0&0&&
 	} 	 $$
 further, the sequence $\eta$ is settled in  the  rightmost column.	Because of being splitting the first three columns,  we can consider this part of the above digram  as a commutative diagram in $\mmod \CX$. Then by getting cokernel of the induced commutative digram in $\mmod \CX$, we obtain a short exact sequence in $\mmod \CX$. In fact, the obtained short exact sequence in $\mmod \CX$ is the image of $\eta$ under the functor  ${}_{\underline{\CY}}\varUpsilon_{\underline{\CX}}$. So we are done. 	
 \end{proof}

 We know by the above proposition the essential image of  ${}_{\underline{\CY}}\varUpsilon_{\underline{\CX}}$ is contained in $\rm{Gprj}\mbox{-}\underline{\CX}$. Denote by ${}_{\underline{\CY}}\hat{\varUpsilon}_{\underline{\CX}}:\mmod \CY\rt \rm{Gprj}\mbox{-}\underline{\CX}$ the induced functor.

 \begin{definition}\label{almsotdense}
 	A functor $F:\mathcal{C}\rt \mathcal{C}'$  is called almost dense if all but  finitely many indecomposable objects, up to isomorphism, are in the essential image of $F.$
 \end{definition}

\begin{proposition}\label{Proposition 4.10}
	Let $\CX$ and $\CY$ be as in Setup \ref{4.17}. Then the following assertions hold.
	\begin{itemize}
		\item [$(i)$] The induced   functor ${}_{\underline{\CY}}\hat{\varUpsilon}_{\underline{\CX}}:\mmod \CY\rt \rm{Gprj}\mbox{-}\underline{\CX}$ is  fully faithful exact. Here assume that $\rm{Gprj}\mbox{-}\underline{\CX}$ gets a natural exact structure from $\mmod \underline{\CX}.$
		\item [$(ii)$] The essential image of ${}_{\underline{\CY}}\hat{\varUpsilon}_{\underline{\CX}}$  contains all indecomposable functors in $\rm{Gprj}\mbox{-}\underline{\CX}$ except the  projective  indecomposable $(-, \underline{X})$, where $X$ is an indecomposable module in  $\CX$ but not in $\CY$.
		\item [$(iii)$]  If $\CX$ is of finite representation type, then ${}_{\underline{\CY}}\hat{\varUpsilon}_{\underline{\CX}}$ is almost dense. 
	\end{itemize}
	\end{proposition}
\begin{proof}
 $(i)$ is a direct consequence of Proposition \ref{Extextensionfunctor} and Construction \ref{Secounddconstruction}. $(ii)$ follows from Theorem \ref{Classification of Gorenstein projective functors} but it remains to show that the representable functors as in the statement are not in the essential image.   Assume to contrary that there is a representable functor $(-, \underline{X})$ as in the statement in the essential image. Then $\tilde{F}={}_{\underline{\CY}}\hat{\varUpsilon}_{\underline{\CX}}(F) \simeq (-, \underline{X})$. Let $(-, N)\rt (-, K)\rt F\rt 0$ be a projective presentation of $F$ in $\mmod \CY.$ Hence  we have the projective presentation $(-, N)\rt (-, K)\rt \tilde{F}\rt 0$ in $\mmod \CX.$ By the isomorphism, we also have the projective presentation  $(-, N)\rt (-, K)\rt (-, \underline{X})\rt 0$ of $(-, \underline{X})$ in $\mmod \CX.$ On the other hand, we know $(-, P_X)\rt (-, X)\rt (-, \underline{X})\rt 0$ is a minimal projective presentation. Considering these  two projective presentations  of $(-, \underline{X})$ in $\mmod \CX$ follows that  $X$ is isomorphic to a direct summand of $K$. But $K$ is in $\CY$, this means $X \in \CY$, that is  a contraction. $(iii)$ is an immediate corollary of $(ii)$.
\end{proof}

 Specializing the above observation to the case when $\CX$ is the subcategory of Gorenstein projective modules over a Gorenstein Artin algebra, we get the following corollaries. For simplicity, we denote the  functor ${}_{\underline{\rm{Gprj}}\mbox{-}\La}\hat{\varUpsilon}_{\underline{\rm{mod}}\mbox{-}\La}$ by ${}_{\CG}\hat{\varUpsilon}_{\La}$.
 
 \begin{corollary}\label{extension functor stable auslander algebra}
 	Let $\La$ be an $n$-Gorenstein algebra. Then
 	\begin{itemize}
 		\item [$(1)$]  
 	 The functor ${}_{\CG}\hat{\varUpsilon}_{\La}: \mmod \underline{\rm{Gprj}}\mbox{-}\La \rt \rm{Gprj}\mbox{-}\underline{\rm{mod}}\mbox{-}\La$ is fully faithful exact. 	 
 	\item[$(2)$] The essential image of ${}_{\CG}\hat{\varUpsilon}_{\La}$ contains all indecomposable functors in $\rm{Gprj}\mbox{-}\underline{\rm{mod}}\mbox{-}\La$ but indecomposable functors $(-, \underline{X})$ such that $X$ is not isomorphic to a  Gorenstein projective indecomposable module. 
 	
 	\item[$(3)$]
 	If $\La$ is of finite representation type, then the functor ${}_{\CG}\hat{\varUpsilon}_{\La}: \mmod \underline{\rm{Gprj}}\mbox{-}\La \rt \rm{Gprj}\mbox{-}\underline{\rm{mod}}\mbox{-}\La$ is almost dense.
 	  	
 		\end{itemize}
 	
 \end{corollary}
 \begin{proof}
 	The corollary is a direct consequence of the above proposition and using this fact that over an $n$-Gorenstein Artin algebra $\La$,  $\Omega^n(\mmod \La)=\rm{Gprj}\mbox{-}\La.$
 	\end{proof}
Let $\La$ be a Gorenstein algebra. Consider the following  composition functor $$\rm{Gprj}\mbox{-}T_2(\La)/ \CV \st{\Psi_{\rm{Gprj}\mbox{-}\La}}\lrt \mmod \underline{\rm{Gprj}}\mbox{-}\La \st{{}_{\CG}\hat{\varUpsilon}_{\La}} \lrt\rm{Gprj}\mbox{-} \underline{\rm{mod}}\mbox{-}\La $$ where $\Psi_{\rm{Gprj}\mbox{-}\La}$ is introduced in Construction \ref{FirstCoonstr}. The composition  is fully faithful by Theorem \ref{Thefirst} and Corollary \ref{extension functor stable auslander algebra}. Further, the corollary follows that the essential image of ${}_{\CG}\hat{\varUpsilon}_{\La} \circ \Psi_{\rm{Gprj}\mbox{-}\La}$ contains all indecomposable functors except the  projective indecomposable functors $(-, \underline{X})$, where $X$ is not Gorenstein projective. Also, if $\La$ is of finite representation type, then the composition functor is almost dense. Hence, in this way we can make a nice connection between the subcategories of Gorenstein projective objects of two different abelian categories.

\begin{corollary}\label{CM-finitensee of triangular matrixes and stable auslander lagebrs}
	Let $\La$ be a Gorenstein algebra of finite representation type. Then the following conditions are equivalent.
	\begin{itemize}
		\item [$(1)$] $T_2(\La)$ is $\rm{CM}$-finite and Gorenstein;
		\item[$(2)$] The stable Cohen-Macaulay Auslander algebra  of $\La$  is a representation-finite self-injective algebra;
		\item[$(3)$] The stable Auslander algebra of $\La$ is  $\rm{CM}$-finite and Gorenstein.
	\end{itemize} 
\end{corollary}
\begin{proof}
The equivalence $(1)\Leftrightarrow (2) $	follows from  Theorem \ref{Thefirst}, \cite[Theorem 3.3.3]{EHS} and a classical result from \cite[Theorem 3.1]{F}, which says that for a triangulated category $\CT$,  $\mmod \CT$ is a Frobenius (abelian) category. The equivalence $(2)\Leftrightarrow (3)$ follows from Corollary \ref{extension functor stable auslander algebra} and \cite[Theorem 3.11]{MT}.
\end{proof}

\section{Almost split sequences}\label{Section 5}
In this section we will show how the functors $\Psi_{\CX}$ and ${}_{\underline{\CY}}\varUpsilon_{\underline{\CX}}$ which are introduced in the previous sections can be used to transfer the almost split sequences. First, we need to recall  the  almost split sequences for the subcategories as follows. Let  $\mathcal{C}$ be a 
 full extension-closed  subcategory of an abelian category $\CA$.  Let $\delta: 0 \rt A \st{f} \rt B\st{g} \rt C\rt 0$ be a short exact sequence in $\CA$. We call $\delta$ a short exact sequence in $\mathcal{C}$ if all whose terms belong to $\mathcal{C}$. A morphism $f:X \rt Y$ in $\mathcal{C}$  is called a {\it proper monomorphism}  if it is a monomorphism in $\CA$ and its cokernel in $\CA$ belongs to $\mathcal{C}$. Dually, one can define a {\it proper epimorphism} in $\mathcal{C}$.  An object $X$ in $\mathcal{C}$ is  called {\it $\Ext$-injective}
if every proper monomorphism $f:X\rt Y$ in $\mathcal{C}$ is  a section; and {\it $\Ext$-projective} 
if every proper epimorphism $g:Z\rt X$ in $\mathcal{C}$ is a retraction. The subcategory $\mathcal{C}$ has enough projectives, if for every object $C$ in $\mathcal{C}$ there is a proper epimorphism $P \rt C$ with $\Ext$-projective object $P$. Dually, one can define the notion of having enough injectives. Since $\mathcal{C}$ is closed under extensions, it gets naturally an exact structure from $\CA$. The above defined notions may also be viewed in the terms of the induced exact structure on $\mathcal{C}$. 

Next, we recall from \cite{AR1, AS} some terminology and facts for the Auslander-Reiten theory.
One says that $f$ in $\mathcal{C}$ is {\it right almost split}  if $f:X\rt Y$ is not a retraction and every non-retraction morphism $g:M\rt Y$ in $\mathcal{C}$ factors through $f$; and {\it minimal right almost split} if $f$ is right minimal and right almost split. In a dual manner, one defines $f$
to be (minimal) left almost split. A short exact sequence 
$$\delta: 0 \rt X \st{f}\rt Y \st{g}\rt Z\rt 0$$ in $\mathcal{C}$
is called {\it almost split}  if $f$ is minimal left almost split and $g$
is minimal right almost split. Since $\delta$ is unique up to isomorphism for $X$ and $Z$, we may write $X=\tau_{\mathcal{C}}Z$ and $Z=\tau^{-1}_{\mathcal{C}}X.$ 
We shall say that $\mathcal{C}$ has {\it right  almost  split}   sequences
if every indecomposable object is either $\Ext$-projective or the ending term of an almost split sequence; dually, $\mathcal{C}$ has {\it left almost split} sequences
if every indecomposable object is either $\Ext$-injective or the starting term of an almost split sequence. We call  $\mathcal{C}$ has {\it almost  split}  sequences
if it has both left and right almost split sequences.\\
 From now on, let $\CX $ be a contravariantly finite and resolving subcategory  of $\mmod \La.$ 

 \subsection{Exchange  between the almost split sequences in $\CS_{\CX}(\La)$ and $\mmod \underline{\CX}$} 
 As we have observed in the beginning of Section \ref{Section 4},  $\mmod \CX$ and $\mmod \underline{\CX}$  have projective covers, and are  Krull-Schmidt categories. 
 
Consider $F \in \mmod \underline{\CX}$ and since $\mmod \CX$ has projective covers,  we can construct a minimal projective resolution for $F$ in $\mmod \CX$. Fix a   minimal projective resolution in $\mmod \CX$ for $F$ as the following
$$\eta_F: 0 \rt (-, A_{F})\st{(-, s_{F})} \rt (-, B_{F})\st{(-, r_{F})}\rt (-, C_{F})\rt F\rt 0.$$
In this way, we have associated  to any $F$ in $ \mmod \underline{\CX}$ with the object $(A_{F}\st{s_{F}}\rt B_{F})$ in $\CS_{\CX}(\La).$

We begin with the following  preliminary result.
\begin{lemma}\label{indecompsable functor and monomorphosms}
$F$ is an indecomposable functor in $\mmod \underline{\CX}$ if and only if $s_{F}$ so is in $\CS_{\CX}(\La).$ 
\end{lemma}
\begin{proof}
	Assume $F$ is indecomposable. By definition of the functor $\Psi_{\CX}$, see Construction \ref{FirstCoonstr}, $\Psi_{\CX}(s_{F})=F$. If $s_{F}$ is not indecomposable, then it must have an indecomposable direct summand of the form   $(X\st{1}\rt X) $ and $(0 \rt X)$. But this means that $\eta_{F}$ is not a minimal projective resolution, a contradiction. To prove the converse, if $F$ is not indecomposable, then there  is  a decomposition $F=F_1\oplus F_2$, $F_1, F_2\neq 0$.  By the uniqueness of the minimal projective resolutions $\eta_{F}\simeq \eta_{F_1}\oplus \eta_{F_2}$,  we get $s_{F}\simeq s_{F_1}\oplus s_{F_2}$, a contradiction. We are done.
	\end{proof}

Note that when  $\CX$ is closed under extensions, this follows $\CS_{\CX}(\La)$ so is in $H(\La)$. Hence  we can talk about the existence of almost split sequences in $\CS_{\CX}(\La)$.

The following lemma gives the structure  of $\Ext$-projective (injective)  indecomposable objects in $\CS_{\CX}(\La)$ which will be helpful later. We need some perpetrations to prove it.
For a given subcategory $\CX$ of $\mmod \La$, we  denote by $\CF_{\CX}(\La)$ the subcategory of $H(\La)$  consisting of all  morphisms $(A\st{f}\rt B)$ satisfying: $(i)$ $f$ is an epimorphism, and $(ii)$ $A, B$ and $\text{Ker} \ f$ belong to $\CX$.

The restrictions of the kernel and cokernel functors (see \cite[Section 1]{RS2}) clearly induce a pair of inverse equivalences between $\CS_{\CX}(\La)$ and $\CF_{\CX}(\La)$.

\begin{lemma}\label{projective objects in monomorphism}
 The following statements occur.

\begin{itemize}
	\item [$(1)$]
  A monomorphism $(A\st{f}\rt B)$ is an  $\Ext$-projective indecomposable object in  $\CS_{\CX}(\La)$ if and only if it is isomorphic to either $(P\st{1}\rt P)$ or $(0 \rt P)$ with  projective indecomposable module  $P$ in $\mmod \La.$
 \item [$(2)$] Assume $\mathcal{X}$ has enough injectives. A monomorphism $(A\st{f}\rt B)$ is an  $\Ext$-injective indecomposable object in $\CS_{\CX}(\La)$ if and only if it is isomorphic to either $(I\st{1}\rt I)$ or $(0 \rt I)$ with $\Ext$-injective  indecomposable  module $I$ in $\CX$.
\end{itemize}
 
\end{lemma} 
\begin{proof}
	$(1)$ Assume $(A\st{f}\rt B)$ is an $\Ext$-projective indecomposable object in the  category $\CS_{\CX}(\La)$. Let $\pi:P\rt A$ and $\pi':Q\rt \text{Cok} \ f$ be the projective covers of $A$ and $\text{Cok} \ f$, respectively. There is a morphism $e:Q \rt B$ such that $de=\pi'$, where $d:B\rt \text{Cok} \ f$ is the canonical epimorphism. We have a proper epimorphism in the following form
	{\footnotesize \[ \xymatrix@R-2pc {  & P\ar[dd]^{l}  & A\ar[dd]~  \\   &  _{\ \ \ \ \ \  \ \ \  }\ar[r]^{\pi}_{[e~~f
				\circ\pi]}  \ar[r]&_{\ \ \ \ \ }{f}   \\ & Q \oplus P & B }\]}
	in $\CS_{\La}(\CX)$, where $l=[0~~1]^t$. The above proper epimorphism in $\CS_{\CX}(\La)$ follows $f$ must be  in one of the forms imposed in the statement. It is also easy to see that the  indecomposable objects  stated in   the statements are $\Ext$-projective in $\CS_{\CX}(\La)$. 
	
	$(2)$ In the similar way of $(1)$, one can show that: An object $(A\st{f}\rt B)$ in $\CF_{\CX}(\La)$ is an indecomposable $\Ext$-injective in  $\CF_{\CX}(\La)$ if and only if it is isomorphic to either $(I\st{1}\rt I)$ or $(I \rt 0)$ with  $\Ext$-injective indecomposable module $I$ in $\CX$. Applying  the kernel functor, defined in the above, and in view of the  characterization of $\Ext$-injective objects in $\CF_{\CX}(\La)$ yield the result.
\end{proof}
\begin{setup}\label{Setup1}
	Let  $\CX \subseteq \mmod \La$ be  contravariantly 
	finite and resolving.  Further, assume $\CX$ has enough injectives and $\CS_{\CX}(\La)$ has almost split sequences.
\end{setup}
In the following construction we shall explain how  the almost split sequences in $\mmod \underline{\CX}$ can be computed by the ones  in $\CS_{\CX}(\La)$. 
\begin{construction}\label{Secound constrution} 
	Let $\CX$ be as in Setup \ref{Setup1}.
 Let $H$ be a  non-projective indecomposable functor in $\mmod \underline{\CX}$. Then by Lemma \ref{indecompsable functor and monomorphosms}, $s_{H}$ so is indecomposable in $\CS_{\CX}(\La)$. Also, it is not a projective object in the exact category $\CS_{\CX}(\La)$, see Lemma \ref{projective objects in monomorphism}. Hence by our assumption there is an almost split sequence in $\CS_{\CX}(\La)$ ending at $s_{H}$, namely, 
	{\footnotesize  \[ \xymatrix@R-2pc {  &  ~ X_1\ar[dd]^{\rm{d}}~   & Z_1\ar[dd]^{h}~  & A_{H}\ar[dd]^{s_{H}} \\ \epsilon: \ \ 0 \ar[r] &  _{ \ \ \ \ } \ar[r]^{\phi_1}_{\phi_2}  &_{\ \ \ \ \ } \ar[r]^{\psi_1}_{\psi_2} _{\ \ \ \ \ }& _{\ \ \ \ \ \ }\ar[r] & 0 \\ & X_2 & Z_2 &B_{H} }\]}
		By expanding the above diagram in $\mmod \La$ we get the following commutative diagram
	$$\xymatrix{& 0 \ar[d] & 0 \ar[d] & 0 \ar[d]&(*) &\\
		0 \ar[r] & X_1 \ar[d]^{\rm{d}} \ar[r]^{\phi_1} & Z_1 \ar[d]^{h}
		\ar[r]^{\psi_1} & A_{H} \ar[d]^{s_{H}} \ar[r] & 0\\
		0 \ar[r] & X_2\ar[d] \ar[r]^{\phi_2} & Z_2
		\ar[d] \ar[r]^{\psi_2} & B_{H} \ar[d] \ar[r] & 0\\
		0 \ar[r] & \text{Cok}\ d \ar[d] \ar[r]^{\mu_1} & \text{Cok} \ h
		\ar[d] \ar[r]^{\mu_2} & C_{H} \ar[d] \ar[r] & 0\\
		& 0  & 0  & 0 & }
	$$
	Since the morphisms $(\text{Id}_{A_H},~ \text{Id}_{A_{H}}):(A_H\st{1}\rt A_H)\rt (A_H\st{s_H}\rt B_H)$ and $(0,~ \text{Id}_{B_H}):(0\rt B_H)\rt (A_H\st{s_H}\rt B_H)$ are non-retraction hence factor through $(\psi_1~~\psi_2)$. Thus, both $\psi_1$ and $\psi_2$ are split. Consequently, $Z_2\simeq X_2\oplus B_{H}$ and $Z_1\simeq X_1\oplus A_{H}.$  Now we show that $\mu_2$ is also split epimorphism. To this do, consider the morphism $(\sigma_1,~~\sigma_2):(\Omega_{\La}(C_{H})\rt P)\rt (A_{H}\st{s_{H}}\rt B_{H})$	obtaining by the following commutative diagram
	$$\xymatrix{
		0 \ar[r] &\Omega_{\La}(\text{Cok}\ s_{H}) \ar[d]^{\sigma_1} \ar[r] & P \ar[d]^{\sigma_2}
		\ar[r] &C_{H} \ar@{=}[d]\ar[r] & 0&\\
		0 \ar[r] & A_{H}  \ar[r]^{s_{H}} & B_{H}
		\ar[r] & C_{H}\ar[r] &0.& \\ 	} 	 $$
	Here, as usual, $P$ is the projective cover of $C_{H}$. But $(\sigma_1,~~\sigma_2)$ is not a retraction, otherwise it leads  to  $H\simeq (-, \underline{D})$, where $D$ is a direct summand of $C_{H}$, a contradiction. Hence $(\sigma_1,~~\sigma_2)$ factors through $(\psi_1,~~\psi_2)$, and consequently $\text{Id}_{C_H}$ factors through $\mu_2$. But this means that $\mu_2$ is a split epimorphism, so the result follows. Applying the Yoneda functor on the diagram $(*)$ and using this observation, as already proved,  the rows are split, then we have the following commutative diagram (applying the isomorphisms ($Z_1\simeq X_1\oplus A_H, Z_2\simeq X_2\oplus B_H, \text{and}  \ \text{Cok} \ h\simeq \text{Cok} \ d \oplus C_H$) due to  the split epimorphisms $\psi_1, \psi_2, \mu_2$, and also with abuse of the notation we denote again by $h$ the corresponding morphism).
	$$\xymatrix{& 0 \ar[d] & 0 \ar[d] & 0 \ar[d]& 0\ar[d]&(**) &\\
		0 \ar[r] & (-,X_1) \ar[d] \ar[r]^{(-, d)} & (-, X_2) \ar[d]
		\ar[r] &(-, \text{Cok}\ d) \ar[d] \ar[r] & F\ar[r]\ar[d]^f&0\\
		0 \ar[r] & (-,X_1\oplus A_{H})  \ar[r]^{(-, h)}\ar[d] & (-, X_2\oplus B_{H})\ar[d]
		\ar[r] & (-,  \text{Cok}\ d\oplus C_H) \ar[r]\ar[d] & G\ar[r]\ar[d]^g& 0\\0\ar[r] &(-, A_{H})\ar[r]^{(-, s_{H})}\ar[d]&(-,  B_{H})\ar[r]\ar[d]&(-, C_{H})\ar[d] \ar[r]&H \ar[r]\ar[d]&0 &\\ &0&0&0&0&&
	} 	 $$
Note that  a renaming of some notation in the above digram  will be given in the last of this construction. 
	We have obtained in the rightmost of the above diagram a  short exact sequence, denote by $S_{H}$, in $\mmod \underline{\CX}$. The first term $F$ in the short exact sequence must be non-zero. Otherwise, because  of indecomposability of $d$ as an object in $\CS_{\CX}(\La)$, since it is the last term of an almost split sequence,  then it would be isomorphic to either  $(0 \rt X)$ or $(X \st{1}\rt X)$ for some  indecomposable $X$ in $\CX.$  Assume the case $d\simeq (X\st{1}\rt X)$ happens. We then have $G\simeq H$. Hence the middle row of $(**)$ gives a projective resolution of $G$, on the other hand, the last row gives the minimal one. Because of being minimal yields  the deleted projective resolution ${\mathbf P}_G$ of $G$, provided in the middle  row of $(**)$, is a direct sum of the deleted minimal projective ${\mathbf P}_{H}$ (in the last row) of  $H$ with some contractible complexes. Namely, $\mathbf{P}_G$ in the category of complexes is isomorphic to such a  decomposition
$$\mathbf{P}_{H}\oplus \mathbf{X}_1\oplus\mathbf{X}_2$$  where for each $i=1, 2$	
	\[ \mathbf{X}_i: \cdots 0 \rt 0 \rt (-, A_i) \st{(-, \text{Id}_{A_i})}\lrt (-, A_i) \rt 0 \rt \cdots, \]	
	for some $A_i $ in $\CX$, where the rightmost term $(-, A_i)$ is located in the degree $1-i$ and  the term $(-, C_{H})$ in $\mathbf{P}_{H}$ is at degree $0$. Since $\text{Cok} \ d=0$ then $A_1$ must be zero. So  by the above decomposition we can deduce the decomposition $h\simeq s_{H}\oplus (A_2 \st{1}\rt A_2)$ in $\CS_{\CX}(\La)$. Comparing domain or codomain of the monomorphisms in the  both sides of the isomorphism, for instance codomain, we get $X_2\oplus B_{H}\simeq B_{H} \oplus A_2$, which implies $A_2\simeq X_2\simeq X.$  Hence $h\simeq s_F\oplus d$. So the middle term  of the short exact sequence  $\epsilon$  is a direct sum of whose ending terms. This means that $\epsilon$ is split, a contradiction. The similar proof works for the case $d\simeq (0 \rt X)$. Hence $F\neq 0.$ Sine $d$ is an indecomposable object in $\CS_{\CX}(\La)$ and $F\neq 0$, hence we  conclude the first row in the digram $(**)$ acts as a minimal projective resolution of $F$ in $\mmod \CX$. So by our convention we identify it by the following fixed minimal projective resolution of $F$
$$ 0 \rt (-, A_F)\st{(-, s_F)} \rt (-, B_F)\st{(-, r_F)}\rt (-, C_F)\rt F\rt 0.$$ 
Therefore,  based on the above facts we can  assume $\epsilon$ has  the following form 
{\footnotesize  \[ \xymatrix@R-2pc {  &_{\ }   A_F\ar[dd]^>>>>{s_F}   &_{ \ \ \ \  \  } A_F\oplus A_{H}\ar[dd]^>>>>{s_{F*{H}}}~  & A_{H}\ar[dd]^>>>>{s_{H}} \\ 0 \ar[r] &  _{ \ \ \ \ } \ar[r]^{[1~~0]^t}_{[1~~0]^t}  &_{\ \ \ \ \ } \ar[r]^{[0~~1]}_{[0~~1]} _{\ \ \ \ \ }&  _{\ \ \ \ \ }\ar[r] & 0 \\ & B_F & B_F\oplus B_{H} & B_{H} }\]}
where $s_{F*{H}}$ is uniquely determined by  $H$ or $F$.
\end{construction}  
\begin{lemma}
Keep in mind all notations used in Construction  \ref{Secound constrution}. The short exact sequence $S_{H}$ is not split.
\end{lemma}
\begin{proof}
	The proof is rather similar to the proof given for showing $F\neq 0$ in the above  construction.
	Assume to the contrary that $S_{H}$ is split. Then $G\simeq F \oplus H.$ According to this fact the minimal projective resolution of $G$ in $\mmod \CX$ is a direct sum of the minimal projective resolutions of $F$ and $H$ given in the first and the last row of the digram $(**)$ in Construction \ref{Secound constrution}, i.e.,
		$$0 \rt (-, A_F\oplus A_{H})\st{(-, s_F\oplus s_{H})}\lrt(-, B_F\oplus B_{H})\st{(-, r_F\oplus r_{H})}\lrt (-, C_F\oplus C_{H})\rt G \rt  0. $$ 
	On the other hand, the sequence in the middle row of $(**)$ is a projective resolution of $G$ (not necessarily  to be minimal). Hence due to the property  of being minimal projective resolution, the deleted minimal  projective resolution of $G$, given in the above,  is a direct summand of the one given in the middle of the diagram $(**)$ in the construction. By translating this fact  in  $\CS_{\CX}(\La)$, we get this decomposition in $\CS_{\CX}(\La)$,  $s_{F*{H}}\simeq s_F\oplus s_{H}\oplus(0 \rt X)\oplus (Y\st{1}\rt Y)$ for some modules $X, Y$ in $\CX$. But $X$ and $Y$ must be zero, otherwise the length of $s_{F*{H}}$, meaning the Jordan-Holder length by considering it as a module over $T_2(\La)$, is greater than the sum of lengths of $s_F$ and $s_H$, that is a contradiction. Summing up, the middle term in $\epsilon$ is a direct sum of whose ending terms, so it means that $\epsilon$ is split and then a contradiction. The proof is now complete.			
\end{proof}

\begin{proposition}\label{AlmostSplitS_H}
	Keep in mind all notations used in Construction  \ref{Secound constrution}. The short exact sequence $S_{H}$ is an almost split sequence.
\end{proposition}
\begin{proof}
By Lemma \ref{indecompsable functor and monomorphosms}, $F$ and $H$ are indecomposable and also by the preceding lemma $S_H$ is non-split. 
Invoking \cite[Theorem 2.14]{AR1}, although it  is stated originally for abelian categories, but however the proof still works for the  exact categories, to prove $\delta$ to be an almost split sequence it is enough to show that $f$  and $g$ respectively are left and right almost split. We will do it only for $g$. Let $v:H'\rt H$ be a non-retraction. Considering it as a morphism in $\mmod \CX$, it can be lifted to their minimal projective resolutions and then returning to $\CS_{\CX}(\La)$ via the Yoneda lemma  to reach  the following morphism  
{\footnotesize \[ \xymatrix@R-2pc {  & A_{H'}\ar[dd]^{s_{H'}}  & A_{H}\ar[dd]~  \\   &  _{\ \ \ \ \ \  \ \ \  }\ar[r]^{v_1}_{v_2}  \ar[r]&_{\ \ \ \ \ }{s_{H}}   \\ & B_{H'} & B_{\rm{H }}}\]}
in $\CS_{\CX}(\La)$ such that $\Psi_{\CX}(v_1, v_2)=v$ The morphism $(v_1,~~v_2)$ is not retraction. Otherwise, it follows $v$ so is, a contradiction. Hence since $\epsilon$ is an almost split sequence,  $(v_1,~~v_2)$ factors thorough the epimorphism of $\epsilon$ via
{\footnotesize \[ \xymatrix@R-2pc {  & A_{H'}\ar[dd]^{s_{H'}}  &A_F\oplus A_{H}\ar[dd]~  \\   &  _{\ \ \ \ \ \  \ \ \  }\ar[r]^>>>>>>>{h'_1}_>>>>>>>{h'_2}  \ar[r]&_{\ \ \ \ \ \ \ \ \ }{s_{F*{H}}}   \\ & B_{H'} & B_F\oplus B_{H} }\]}
Now by applying the functor $\Psi_{\CX}$ on such a factorization, we see that the morphism $v$ factors through $g$ via $\Psi_{\CX}(h_1',~~h_2')$, as required. So we are done.
\end{proof}

As we have observed in the above a recipe  for constructing the almost split sequence in $\mmod \underline{\CX}$ for a  given non-projective indecomposable functor  $H$ of $\mmod \underline{\CX}$ is  via computing the almost split sequence  ending at object $s_{H}$ in $\CS_{\CX}(\La)$. Dually, we can do a similar process for a given non-injective indecomposable  functor $F$ of $\mmod \underline{\CX}$ to compute the almost split sequence in $\mmod \underline{\CX}$ starting at $F$. A direct consequence of these observations is the following result.
\begin{proposition}\label{The exsitence of almosat split via monomorphism catgeories}
	Let $\CX$ be the same as in Set up \ref{Setup1}. Then, $\mmod \underline{\CX}$ has almost split sequences.
\end{proposition}

Let us in continue give some examples of subcategories satisfying the conditions of Set up \ref{Setup1}. According to a theorem of Auslander and smal{\o} (\cite[Theorem 2.4]{AS}), one way to show that an extension-closed subcategory (of finitely generated modules over some Artin algebra) to have almost split sequences is to prove that it is functorially finite. When $\CX=\mmod \La$, in \cite{RS2}  was  proved that the subcategory  $\CS(\La)$ is functorially finite in $H(\La).$ 
So by the above-mentioned fact $\CS(\La)$ has almost split sequences.  In fact in \cite{RS2}, left and right minimal $\CS(\La)$-approximations for any object in $H(\La)$  are computed explicitly. These computations provide a tool for computing the almost split sequences in $\CS(\La)$. \\
Another example is $\CS_{\rm{Gprj}\mbox{-}\La}(\La)$ when $\La$ is Gorenstein. In fact, the Gorensteiness of $\La$ implies that $T_2(\La)$ is so, see \cite[Theorem 3.3.3]{EHS}. Hence, $\CS_{\rm{Gprj}\mbox{-}\La}(\La)=\rm{Gprj}\mbox{-}T_2(\La)$, see Lemma \ref{Gorenproj charecte}, is contravariantly finite in $\mmod T_2(\La)$ by \cite[Theorem 2.10]{Ho}. Thanks to \cite[Corollary 0.3]{KS}, we observe that $\CS_{\rm{Gprj}\mbox{-}\La}(\La)$  is functorially finite. Therefore, by Auslander-smal{\o}' s result, it has almost split sequences. In this way we obtain another subcategory satisfying Set up \ref{Setup1}.

\subsection{Exchange  between the almost split sequences in $\mmod \underline{\CX\cap\rm{Gprj}\mbox{-}\La}$ and $ \rm{Gprj}\mbox{-}\underline{\CX}$}
 As we have seen in the previous subsection a  connection between the almost split sequences in $\mmod \underline{\CX}$ and in $\CS_{\CX}(\La)$, where $\CX$ satisfying the conditions of Set up \ref{Setup1}, is given. Now let $\CX$ and $\CY=\CX\cap\rm{Gprj}\mbox{-}\La$ be the same as in Setup \ref{4.17}. We will give a similar connection between the almost split sequences of $\rm{Gprj}\mbox{-} \underline{\CX}$,  the subcategory of Gorenstein projective functors in $\mmod \underline{\CX}$,  and the ones of  $\mmod \underline{\CY}$.	

 Recall from the previous section that there is the  functor ${}_{\underline{\CY}}\hat{\varUpsilon}_{\underline{\CX}}:\mmod \underline{\CY}\rt \rm{Gprj}\mbox{-} \underline{\CX}$ which is fully faithful exact. For simplicity, set $\tilde{F }:= {}_{\underline{\CY}}\hat{\varUpsilon}_{\underline{\CX}}(F)$ for any $F$ in $\mmod \underline{\CY}$ and also $\tilde{f}:={}_{\underline{\CY}}\hat{\varUpsilon}_{\underline{\CX}}(f)$ for any morphism $f$ in $\mmod \underline{\CY}$. In the next results, we shall show the  functor ${}_{\underline{\CY}}\hat{\varUpsilon}_{\underline{\CX}}$ preserves the almost split sequences. 
\begin{proposition}\label{AlmsotsplitXY}
	Let $\CX$ and $\CY$ be the same as in Set up \ref{4.17}. Assume $\eta: 0 \rt F \st{f}\rt G \st{g}\rt H\rt 0 $ is a short exact sequence in $\mmod \underline{\CY}$. Then $\eta$  is an almost split sequence in $\mmod \underline{\CY}$ if and only if  its image under  ${}_{\underline{\CY}}\hat{\varUpsilon}_{\underline{\CX}}$
	$$\tilde{\eta}: 0 \rt \tilde{F } \st{\tilde{f}} \rt  \tilde{G}\st{\tilde{g}}\rt \tilde{H} \rt 0$$
so 	is 
 in  $\rm{Gprj}\mbox{-} \underline{\CX}.$
\end{proposition} 
 \begin{proof}
 	We only prove the  ``if'' part.
By Proposition \ref{Extextensionfunctor}, $\tilde{\eta}$ is a short exact sequence. Also, since ${}_{\underline{\CY}}\hat{\varUpsilon}_{\underline{\CX}}$ is fully faithful, $\tilde{\eta}$ is non-split with indecomposable ending terms. So analogue to the proof of Proposition \ref{AlmostSplitS_H}, it suffices to show that   $\tilde{f}$ and $\tilde{g}$ are respectively left and right almost split. We will do it only for $\tilde{g}$. Let $h: D \rt \tilde{H}$ be a non-retraction. We may assume $D$ is indecomposable. If $D=\tilde{V}$ for  some $V$ in $\mmod \underline{\CY}$, then $h=\tilde{d}$ for some $d:V \rt H$. The morphism $d$ can not be a retraction because of being fully faithful of ${}_{\underline{\CY}}\hat{\varUpsilon}_{\underline{\CX}}$. Hence it factors through $g$. Then by applying  the functor ${}_{\underline{\CY}}\hat{\varUpsilon}_{\underline{\CX}}$ on the factorization we obtain the desired factorization of $h$ through $\tilde{g}$.  If $D$ does not lie in the essential image of ${}_{\underline{\CY}}\hat{\varUpsilon}_{\underline{\CX}}$, then $D\simeq (-, \underline{X})$ for some $X$ not being  in $\CY$, see Proposition \ref{Proposition 4.10}. In this case,  the result clearly follows since $D$ is projective. 
 \end{proof}

 By putting together Proposition \ref{AlmostSplitS_H} and Proposition \ref{AlmsotsplitXY} we get the next result.
 \begin{theorem}\label{Translation}
 	Let $\CX$ and $\CY$ be the same as in Set up \ref{4.17} and further $\CY$ satisfies the conditions in Set up \ref{Setup1}. Assume  $G$ is a non-projective indecomposable functor in  $\rm{Gprj} \mbox{-}\underline{\CX}$. Then,  there is a non-projective indecomposable functor $F$ in $\mmod \underline{\CY}$ such that $G={}_{\underline{\CY}}\hat{\varUpsilon}_{\underline{\CX}}(F)$ and
 	
 	$$\tau_{\rm{Gprj}\mbox{-}\underline{\CX}} (G)\simeq {}_{\underline{\CY}}\hat{\varUpsilon}_{\underline{\CX}}(\tau_{\underline{\CY}}(F))\simeq {}_{\underline{\CY}}\hat{\varUpsilon}_{\underline{\CX}}\circ \Psi_{\CY}(\tau_{\CS_{\CY}(\La)}(s_F)),$$ 
 	where $\tau_{\underline{\CY}}$ denotes the Auslander-Reiten translation in $\mmod \underline{\CY}.$
 \end{theorem}
 As an application of our result in conjunction with \cite[Corollary 6.5]{RS2} we have:
\begin{corollary}
	Let $\La$ be a commutative   Nakayama algebra.  Let $\Gamma$ denote the stable Auslander algebra of $\La$, i.e. $\Gamma :=\underline{\rm{End}}(M),$ where $M$ is a representation generator of $\mmod \La.$ Then $\tau^{6}_{\Gamma}(N)\simeq N $ for each non-projective indecomposable $N $ in $\mmod \Gamma.$	
\end{corollary}

\section{Auslander-Reiten quivers}
This section in fact is established as a direct consequence of the results of the previous section  for   making a  connection between the Auslander-Reiten quivers of the different categories. The notion of the Auslander-Reiten quiver of a category is a convenient combinatorial tool to encode the  almost split sequences. Let us first recall the notion of an  Auslander-Reiten quiver.   Let  $\mathcal{C}$ be a 
full extension-closed   subcategory   of an abelian  Krull-Schmidt category $\CA$. It is defined by Auslander and Reiten for $\mathcal{C}$ an associated valued quiver $\Gamma_{\mathcal{C}}=(\Gamma^0_{\mathcal{C}}, \Gamma^1_{\mathcal{C}})$, called  the Auslander-Reiten quiver of $\mathcal{C}$, as follows: The vertices  $\Gamma^0_{\mathcal{C}}$ are in one to one correspondence with isomorphism classes of the objects in $\mathcal{C}$, are usually denoted by $[M]$, sometimes simply $M$, for an  indecomposable object $M$ in $\mathcal{C}$. There is an arrow $[M]\rt [N]$ in $\Gamma^1_{\mathcal{C}}$ with valuation $(a, b)$ if there are a minimal right almost split morphism  $M^a\oplus X\rt N$ such that $X$ has no direct summand isomorphic to $M$ in $\mathcal{C}$, and a minimal left almost split morphism $M \rt N^b\oplus Y$ in $\mathcal{C}$. If the valuation of an arrow is trivial $(1, 1)$, we  only write an arrow.
 Assume that the $\Ext$-projective objects in $\mathcal{C}$ and $\Ext$-injective objects in $\mathcal{C}$ coincide. We denote by $\Gamma^s_{\mathcal{C}}$ the valued quiver, called the stable Auslander-Reiten quiver of $\mathcal{C}$, obtained by removing all vertices corresponding to    $\Ext$-projective and $\Ext$-injective indecomposable objects in $\mathcal{C}$ and the arrows attached to them.\\
In order to simplify our notation for when $\mathcal{C}=\mmod \mathcal{D}$, we denote $\Gamma_{\mathcal{D}}$ instead of $\Gamma_{\mmod \mathcal{D}}$. In particular, $\Gamma_{\La}$ stands for  $\Gamma_{\mmod \La}$.
Since an essential source for minimal right (left) almost split morphisms  is the almost split sequences, then based on the definition of the Auslander-Reiten quivers,  an immediate consequence of our result in Section \ref{Section 5} is:

\begin{theorem}\label{main4}
	The following statements hold.
	\begin{itemize}
		\item [$(a)$] Let $\CX$ be the same as in Set up \ref{Setup1}. Then 
		\begin{itemize}
			\item[$(1)$] The  valued quiver $\Gamma_{\underline{\CX}}$ is a full valued subquiver of   $\Gamma_{\CS_{\CX}(\La)}$.
			\item[$(2)$] If $\CX$ is of finite representation type, then 
		 $\Gamma_{ \underline{\CX}}$ differs  only finitely many vertices with $\Gamma_{\CS_{\CX}(\La)}$.  In particular,  $$|\Gamma^0_{\CS_{\CX}(\La)}|=2|\rm{Ind}\mbox{-}\CX|+|\Gamma^0_{\underline{\CX}}|.$$
	
		\end{itemize} 
	\item [$(b)$]
	Let $\CX$ and $\CY=\CX\cap\rm{Gprj}\mbox{-}\La$ be the same as in Set up \ref{4.17} and further $\CY$ satisfies the conditions in Set up \ref{Setup1}. Then
	\begin{itemize} \item [$(1)$] The valued quiver  $\Gamma_{\underline{\CY}}$ is a full valued  subquiver of $\Gamma_{\rm{Gprj}\mbox{-} \underline{\CX}}$ and $\Gamma_{\CS_{\CY}(\La)}$, i.e.,
	$$\Gamma_{\CS_{\CY}(\La)}\hookleftarrow \Gamma_{ \underline{\CY}}\hookrightarrow \Gamma_{\rm{Gprj}\mbox{-} \underline{\CX}}.$$
	\item[$(2)$] If  $\CX$ is of finite representation type, then $\Gamma_{ \underline{\CY}}$ differs  only finitely many vertices with $\Gamma_{\CS_{\CY}(\La)}$  and $\Gamma_{\rm{Gprj}\mbox{-} \underline{\CX}}.$ In particular,  $$|\Gamma^0_{\CS_{\CY}(\La)}|=2|\rm{Ind}\mbox{-}\CY|+|\Gamma^0_{\underline{\CY}}|\ \text{and} \ 
	|\Gamma^0_{\rm{Gprj}\mbox{-} \underline{\CX}}|=|\Gamma^0_{ \underline{\CY}}|+|\rm{Ind}\mbox{-}(\CX \setminus\CY)|,$$ 
	where $\CX \setminus\CY$ denotes the subcategory of objects in $\CX$ but not in $\CY.$ 	
	\end{itemize}
\end{itemize}		
	\end{theorem}
\begin{proof}
	$(a)$: the part $(1)$ follows from Proposition \ref{AlmostSplitS_H} and its dual. The part $(2)$ is a consequence of the equivalence $\CS_{\CX}(\La)/\CV\simeq \mmod \underline{\CX},$   where $\CV$ is additive closure of the  objects of the form  $(X \st{1}\rt X )$ or $ (0 \rt X)$, which $X$ runs through   of objects in $\mathcal{X}$. We observe by the equivalence only $2|\rm{Ind}\mbox{-}\CX|$ indecomposable objects, up to isomorphism, in $\CS_{\CX}(\La)$ vanish. $(b):$ the part $(1)$ follows from $(a)(1)$ and Proposition \ref{AlmsotsplitXY}. The first equality follows from part $(a)(2)$, and the latter equality follows from Corollary \ref{extension functor stable auslander algebra}.
\end{proof}
We specialize for when  $\CX=\mmod \La$ and $\La$ is a Gorenstein algebra. In this case $\CY$ is equal to $\rm{Gprj}\mbox{-}\La.$
\begin{theorem}\label{main5}
	Let $\La$ be a Gorenstein algebra. The Auslander-Reiten quiver 	   $\Gamma_{\underline{\rm{Gprj}}\mbox{-}\La}$ of $\rm{mod}\mbox{-} \underline{\rm{Gprj}}\mbox{-}\La$ is  embedded into $\Gamma_{\rm{Gprj}\mbox{-}T_2(\La)}$ and $\Gamma_{\rm{Gprj}\mbox{-} \underline{\rm{mod}}\mbox{-}\La}$ as a full valued subquiver, i.e.,
	$$\Gamma_{\rm{Gprj}\mbox{-}T_2(\La)}\hookleftarrow \Gamma_{\underline{\rm{Gprj}}\mbox{-}\La}\hookrightarrow \Gamma_{\rm{Gprj}\mbox{-} \underline{\rm{mod}}\mbox{-}\La}.$$	
	 Moreover, the stable Auslander-Reiten quivers $\Gamma^s_{\underline{\rm{Gprj}}\mbox{-} \La}=\Gamma^s_{\rm{Gprj}\mbox{-} \underline{\rm{mod}}\mbox{-}\La}$.  
\end{theorem}
The above theorem when is more interesting that $\La$ is  of finite representation type. In fact, the above results say for  computing the Auslander-Reiten quiver of the subcategory of Gorenstein projective modules over the associated stable Auslander algebra $\rm{Aus}(\underline{mod}\mbox{-}\La)$ and the corresponding  triangular matrix algebra $T_2(\La)$ is enough to compute the Auslander-Reiten quiver of the corresponding stable Cohen-Macaulay Auslander algebra $\rm{Aus}(\underline{Gprj}\mbox{-}\La)$ and then adding finitely many vertices. As it will be seen in below the finite remaining part is not difficult to be added. Since $\rm{Aus}(\underline{Gprj}\mbox{-}\La)$ is  self-injective there are many results about the shape of the Auslander-Reiten quivers of   self-injective algebras in the literature,  which can be transfered to the corresponding subcategories of Gorenstein projective modules.

 For the triangular matrix case, in the following we will provide the structure of the  almost split sequences in $\CS_{\CX}(\La)$ ending at or starting from the indecomposable objects of the form $(X\st{1}\rt X)$ or $(0 \rt X)$. It is  helpful for getting completely $\Gamma_{\CS_{\CX}(\La)}$ while we know the $\Gamma_{ \underline{\CX}}.$

\begin{lemma}\label{AlmostSplittrivialmonomorphisms}
	Let $\CX$ be the same as in Set up \ref{Setup1}. Let $0 \rt  A\st{f} \rt B\st{g} \rt C \rt 0$  be an almost split sequence in $\CX.$ Then
	\begin{itemize}
		\item[$(1)$] The almost split sequence in $\CS_{\CX}(\La)$ ending at $(0\rt C)$ has the form 
		{\footnotesize  \[ \xymatrix@R-2pc {  &  ~ A\ar[dd]^{1}~   & A\ar[dd]^{f}~  & 0\ar[dd] \\ 0 \ar[r] &  _{ \ \ \ \ } \ar[r]^{1}_{f}  &_{\ \ \ \ \ } \ar[r]^{0}_{g} _{\ \ \ \ \ }&  _{\ \ \ \ \ }\ar[r] & 0. \\ & A & B &C}\]}
		\item [$(2)$] Let $e:A \rt I$ be a left minimal morphism with $\rm{Ext}$-injective $I$. Then the almost split sequence ending at $(C\st{1}\rt C)$ in $\CS_{\CX}(\La)$ has the form
		{\footnotesize  \[ \xymatrix@R-2pc {  &  ~ A\ar[dd]^{e}~~   & B\ar[dd]^{h}~  & C\ar[dd]^1 \\ 0 \ar[r] &  _{ \ \ \ \ } \ar[r]^{f}_{[1~~0]^t}  &_{\ \ \ \ \ } \ar[r]^{g}_{[0~~1]} _{\ \ \ \ \ }&  _{\ \ \ \ \ }\ar[r] & 0, \\ & I & I \oplus C &C}\]} 	  	 	
		where $h$ is the map $[e'~~g]^t$ with $e':B \rt I$ is an extension of $e.$
		\item [$(3)$] Let $b:P_C\rt C$ be the projective cover of $C$. Then the almost split sequence starting  at $(0 \rt A)$ in $\CS_{\CX}(\La)$ has the form 
			{\footnotesize  \[ \xymatrix@R-2pc {  &  ~ 0\ar[dd]   & \Omega_{\La}(C)\ar[dd]^{h}~  & \Omega_{\La}(C)\ar[dd]^i \\ 0 \ar[r] &  _{ \ \ \ \ } \ar[r]^{0}_{[1~~0]^t}  &_{\ \ \ \ \ } \ar[r]^{1}_{[0~~1]} _{\ \ \ \ \ }&  _{\ \ \ \ \ }\ar[r] & 0, \\ & A & A \oplus P_C &P_C}\]}
		where $h$ is  the kernel of morphism $[f~~b']:A\oplus P_C\rt B$, here $b'$ is a lifting of $b$ to $g$. 
		 
	\end{itemize}
\end{lemma}
\begin{proof}
	The assertions $(1)$ and $(2)$ are  proved in \cite[Proposition 7.1]{RS2} for when $\CX=\mmod \La.$ The argument given in there still works for our setting. To prove the  assertion $(3)$, with \cite[Proposition 7.4]{RS2}, which still works in our setting,  the almost split sequence in $\CF_{\CX}(\La)$ starting at $(A \st{1}\rt A)$ has the form 
		{\footnotesize  \[ \xymatrix@R-2pc {  &  ~ A\ar[dd]^{1}~   & A\oplus P_C\ar[dd]^{d}~  & P_C\ar[dd]^b \\ 0 \ar[r] &  _{ \ \ \ \ } \ar[r]^{[1~~0]^t}_{f}  &_{\ \ \ \ \ } \ar[r]^{[0~~1]}_{g} _{\ \ \ \ \ }&  _{\ \ \ \ \ }\ar[r] & 0 \\ & A & B &C}\]}
	where $d=[f~~b']$. The short exact sequence stated in $(3)$ is only obtained by applying the kernel functor on the  above one. Since an equivalence preserves  the almost split sequences so one can conclude the desired result. 	 
\end{proof}

In the rest of the section we present some examples to show explicitly that how our results help to compute the Auslander-Reiten quivers of the subcategories of Gorenstein projective modules.\\ 
In the sequel $k$ denotes an algebraic closed filed.

In the next example we show how by having the Auslander-Reiten quiver of the stable Auslander algebra of $\La$ one can build the Auslander-Reiten quiver of the submodule category $\CS(\La)$. The point here is to add the remaining vertices corresponding to isomorphism classes  of the  indecomposable objects in $\CS({\La})$ being  of the form  $(X\st{1}\rt X)$ and $(0\rt X)$  for some indecomposable $X$ in $\mmod \La.$ To  this aim,  Lemma \ref{AlmostSplittrivialmonomorphisms} will be helpful. \\
 For an indecomposable object $(A\st{f}\rt B)$ in $\CS(\La)$, we use the notation ``$AB_f$'' to show the corresponding vertex of isomorphism class of $(A\st{f}\rt B)$ in $\Gamma_{\CS(\La)}$. If no ambiguity may arise, we omit ``$f$''. Hence by our convention, $AA_{1}$ and $0A_0$, or simply  $AA$ and $0A$,  respectively always show the vertices in $\Gamma_{\CS(\La)}$ associated with  isomorphism classes  of the objects $(A\st{1}\rt A)$ and $(0\rt A)$. Moreover, we sometimes use bracket  to distinguish ``$A$'' and ``$B$''.

\begin{example} \label{Example1} Let $\La$ be the path algebra of the quiver $A_3:1\rt 2 \rt 3$. The Auslander-Reiten quiver $\Gamma_{\La}$ of $\La$ is 	
	$$
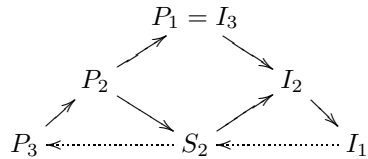
\xymatrix@-5mm{
	&&&&P_1=I_3\ar[dr]\\
	&&&P_2\ar[dr]\ar[ur]&&I_2\ar[dr]\\
	&&P_3\ar[ur]&&S_2\ar[ur]\ar@{..>}[ll]&&I_1\ar@{..>}[ll]	}
$$
As usual, the $P_i, I_i$ and $S_i$, respectively, show the  indecomposable projective, injective and simple modules corresponding to the vertex $i$. Then the stable Auslander algebra $A=\rm{Aus}(\underline{mod}\mbox{-}\La)$ is given by the following quiver
	$$\hskip .5cm \xymatrix@-4mm{
		{I_1}\ar [r]^{\alpha}&{I_2}\ar [r]^{\beta}& {S_2}}\hskip .5cm$$	
	bound by $\alpha \beta=0.$ The Auslander-Reiten quiver of $A$ is of the form
	$$
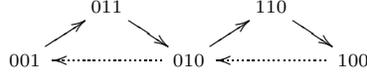
\xymatrix@-5mm{		
	&{\bsm011\esm}\ar[dr]&&{\bsm110\esm}\ar[dr]&\\
	{\bsm001\esm}\ar[ur]&&{\bsm010\esm}\ar[ur]\ar@{..>}[ll]&&{\bsm100\esm}\ar@{..>}[ll]	}
$$
Hence the indecomposable functors  $$(-, S_2)/\rm{rad}(-, S_2), (-, \underline{I_2}), (-, I_2)/\rm{rad}(-, S_2), (-, \underline{I_1}) \ \text{and} \ (-, I_1)/\rm{rad}(-, I_1)$$ in $\mmod\underline{ \rm{mod}}\mbox{-} \La$ are mapped  under the  evaluation functor to the indecomposable modules in  $\mmod A$ with  dimension vectors  ${\bsm001\esm}, {\bsm011\esm}, {\bsm010\esm}, {\bsm110\esm}$ and ${\bsm100\esm}$, respectively. For transforming our results to $\CS(\La)$ we need to compute the minimal projective resolutions of the indecomposable functors in $\mmod\mmod \La.$ The case for a  projective indecomposable functor is clear it is induced by getting  minimal projective resolution in $\mmod \La$, and for a simple functor it is obtained by computing an  almost split sequence in $\mmod \La.$ Hence by help of the $\Gamma_{\La}$ we have the following minimal projective resolutions in $\mmod \mmod \La$:
$$0 \rt (-, P_2)\rt (-, P_1\oplus S_2)\rt (-, I_2)\rt (-, I_2)/\rm{rad}(-, I_2)\rt 0$$
$$0 \rt (-, P_3)\rt (-, P_2)\rt (-, S_2)\rt (-, S_2)/\rm{rad}(-, S_2)\rt 0 $$
$$0 \rt (-, S_2)\rt (-, I_2)\rt (-, I_1)\rt(-, I_1)/\rm{rad}(-, I_1) \rt 0$$
$$0 \rt (-, S_3)\rt (-, P_1)\rt (-, I_2)\rt (-, \underline{I_2})\rt 0$$
$$0 \rt (-, P_2)\rt (-, P_1)\rt (-, I_1)\rt (-, \underline{I_1})\rt 0.$$
We can consider $\Gamma_A$ as a full subquiver of $\Gamma_{\CS(\La)}$ by Theorem \ref{main4}. Based on   the embedding any vertex in $\Gamma_A$ is corresponded to a vertex of $\Ga_{\CS(\La)}$ such that the associated indecomposable objects of the vertices  are related by the minimal projective resolution as listed in the above.   Now by adding the remaining part in view  of  Lemma \ref{AlmostSplittrivialmonomorphisms}, the Auslander-Reiten quiver of $\CS(\La)$ as follows.

	$$
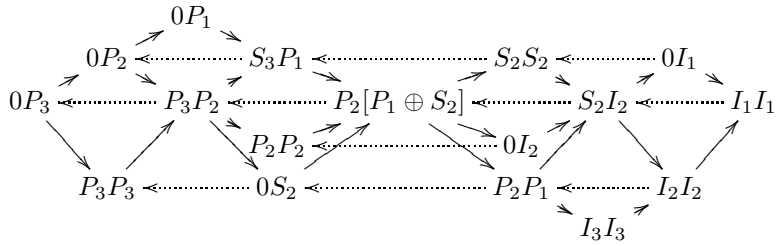
\xymatrix@R=0.05cm@C=0.15cm{ && 0P_1\ar[dr]& & & &
	 & & & & \\
		&0P_2\ar[dr]\ar[ur] & &
		S_3P_1\ar[dr]\ar@{.>}[ll] & &
		S_2S_2\ar[dr]\ar@{.>}[ll]
		& & 0I_1\ar[dr]\ar@{.>}[ll]
		& &
	&	& \\
	0P_3\ar[ur]\ar[ddr]	&& P_3P_2\ar[dr]\ar[ur]\ar[ddr]\ar@{.>}[ll] & &
		P_2[P_1\oplus S_2]\ar[ddr]\ar[dr]\ar[ur]\ar@{.>}[ll] & &
		S_2I_2\ar[ur]\ar[ddr]\ar@{.>}[ll]
		& & I_1I_1\ar@{.>}[ll] & &
		 \\&
		 &&
		P_2P_2\ar[ur] & &
		0I_2\ar[ur]\ar@{.>}[ll] & &
		 & &
		& \\ & P_3P_3\ar[uur]
		& & 0S_2\ar[uur]\ar@{.>}[ll]&
		 &P_2P_1\ar[uur]\ar[dr]\ar@{.>}[ll] &
		 &I_2I_2\ar[uur]\ar@{.>}[ll] &
		 & &
	  \\ &
		& & &		
		& &I_3I_3\ar[ur]  & &		
	& & }$$
Note that the vertices in the above diagram (or else where) which  are presented by the same symbol have to be identified.
			\end{example}

\begin{example}
	Let $\La=T_2(k[x]/(x^2))$, $k[x]$ be   the polynomial ring in one variable $x$ with coefficients in $k$. Then $\La$ is a $1$-Gorenstein algebra of finite representation type. In fact, it is an easy  example of  algebras considered in \cite{GLS}.  The algebra $\La$ is given by the quiver 
	
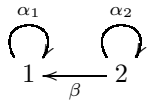
\[
	\xymatrix{
		1 \ar@(ul,ur)^{\alpha_1}& 2 \ar[l]^{\beta}\ar@(ul,ur)^{\alpha_2}
	}
	\]
with relations $\alpha_1^2=\alpha^2_2=0$ and $\alpha_2\beta=\beta\alpha_1.$ Set $A=\rm{Aus}(\underline{mod}\mbox{-}\La)$ and $B=\rm{Aus}(\underline{\rm{Gprj}}\mbox{-}\La)$. Due to the computation given in \cite[Example 13.5]{GLS}, the Auslander-Reiten quiver $\Gamma_{\La}$ of $\La$ looks as the following

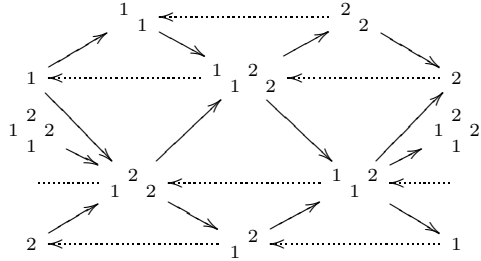
\[
\xymatrix@R=0.01cm@C=0.10cm{
	&&{\bsm 1&\\&1 \esm}\ar[ddrr]&&&&
	{\bsm 2&\\&2 \esm}\ar@{.>}[llll]\ar[ddrr]
	\\
	&&&&&&&&
	\\
{\bsm 1 \esm}\ar[uurr]\ar[ddrr]&&&&
	{\bsm 1&&2\\&1&&2 \esm}\ar[uurr]\ar[ddrr]\ar@{.>}[llll]&&&&
	{\bsm 2 \esm}\ar@{.>}[llll]
	\\
	{\bsm &2\\1&&2\\&1 \esm} \ar[drr]
	&&&&&&&&
	{\bsm &2\\1&&2\\&1 \esm}
	\\
	&&{\bsm &2\\1&&2 \esm}\ar[ddrr]\ar[uurr]\ar@{.}[ll]&&&&
{\bsm 1&&2\\&1 \esm}\ar@{.>}[llll]\ar[ddrr]\ar[uurr]\ar[rru]
	&&\ar@{.>}[ll]
	\\
	&&&&&&&&
	\\
	{\bsm 2 \esm}\ar[uurr]&&&&
	{\bsm &2\\1 \esm}\ar[uurr]\ar@{.>}[llll]&&&&
	{\bsm 1 \esm}\ar@{.>}[llll]
}
\] 
where the  vertices are displayed by the composition series. So the algebra  $A$ is given by the following quiver with the mesh relations (indicated by dashed lines) and the two vertices in the leftmost column have to be identified with the two vertices  in the rightmost column with respect to  the same sign.
	
		
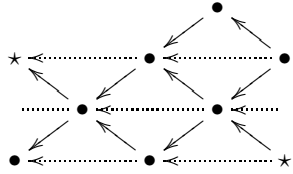
\[
		\xymatrix@R=0.05cm@C=0.15cm{
			&&&&&&
		\bullet \ar[ddll]
			\\
			&&&&&&&&
			\\
			\star&&&&
		\bullet\ar[ddll]\ar@{.>}[llll]&&&&
			\bullet\ar[uull]\ar[ddll]\ar@{.>}[llll]
			\\		
			&&&&&&&&			
			\\
			&&\bullet\ar[ddll]\ar[uull]\ar@{.}[ll]&&&&
			\bullet\ar[ddll]\ar[uull]\ar@{.>}[llll]
			&&\ar@{.>}[ll]
			\\
			&&&&&&&&
			\\
			\bullet&&&&
			\bullet\ar[uull] \ar@{.>}[llll]&&&&
			\star\ar[uull]\ar@{.>}[llll]		}
		\] 

The indecomposable modules with composition series ${\bsm &2\\1 \esm}, {\bsm 1&&2\\&1 \esm}$ and ${\bsm 1 \esm}$ are Gorenstein projective modules; take the  Gorenstein projective indecomposable module  $G_1, G_2$ and $G_3$ respectively with those composition series. Let $P_1$ and $P_2$ denote respectively  the  projective indecomposable modules  corresponding to the vertex $1$ and $2$ in the bound quiver of $\La$. The Auslander-Reiten quiver $\Gamma_{\rm{Gprj}\mbox{-}\La}$ of the subcategory $\rm{Gprj}\mbox{-}\La$ is given as the following

	
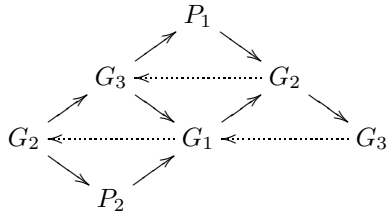
\[
\xymatrix@R=0.05cm@C=0.15cm{
	&&&&&&
	P_1 \ar[ddrr]&&
	\\
	&&&&&&&&&&
	\\
&&&&
	G_3\ar[uurr]\ar[ddrr]&&&&
	G_2\ar[ddrr]\ar@{.>}[llll]&&
	\\		
	&&&&&&&&			
	\\
	&&G_2\ar[ddrr]\ar[uurr]&&&&
	G_1\ar[uurr]\ar@{.>}[llll]
	&& &&G_3\ar@{.>}[llll]
	\\
	&&&&&&&&&&
	\\
	&&&&
	P_2\ar[uurr] &&&&
	&&		}
\] 

In view of $\rm{Gprj}\mbox{-}\La$, 	the  stable Cohen-Macaulay Auslander algebra $B$ of $\La$  is given by the quiver 
	$$\hskip .5cm 
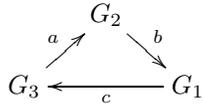
\xymatrix@-4mm{
		&G_2\ar @{<-}[dl]_a\ar[dr]^b\\
		G_3&& G_1\ar [ll]^{c}}\hskip .5cm$$
	with the relations $ab=bc=ca=0$. 	 The Auslander-Reiten quiver of $B$ is of the form 
	$$
	
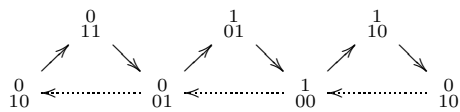
\xymatrix@-5mm{		
		&{\bsm0\\11\esm}\ar[dr]&&{\bsm1\\01\esm}\ar[dr]&&{\bsm1\\10\esm}\ar[dr]\\
		{\bsm0\\10\esm}\ar[ur]&&{\bsm0\\01\esm}\ar[ur]\ar@{.>}[ll]&&{\bsm1\\00\esm}\ar[ur]\ar@{.>}[ll]&&\ar@{.>}[ll]{\bsm0\\10\esm}	}
	$$
	where the vertices are displayed by the dimension vectors. Hence the dimension vectors ${\bsm0\\10\esm}, {\bsm0\\01\esm}$ and ${\bsm0\\10\esm}$ present respectively  the simple functors $S_{G_3}=(-,G_3)/\rm{rad}(-, G_3), S_{G_1}=(-,G_1)/\rm{rad}(-, G_1)$ and $S_{G_2}=(-,G_2)/\rm{rad}(-, G_2)$; also the dimension vectors ${\bsm0\\11\esm}, {\bsm1\\01\esm}$ and ${\bsm1\\10\esm}$ present respectively the  projective indecomposable functors $(-, \underline{G_1}), (-, \underline{G_2})$ and $(-, \underline{G_3})$.
	Hence the  Auslander-Reiten quiver $\Gamma_{\rm{Gprj}\mbox{-}A}$  of the subcategory of  Gorenstein projective modules over $A$, by Theorem \ref{main5},  contains the above quiver. Take the indecomposable modules $M, N, T, U$ in $\mmod \La$ such that their composition series respectively are presented as  ${\bsm 2 \esm}, {\bsm &2&\\1&&2 \esm}, {\bsm 1&&2\\&1&&2 \esm}$ and $	{\bsm 2&\\&2 \esm}$. In fact, they are those of indecomposable modules which are not included in the image of the  functor ${}_{\CG}\hat{\varUpsilon}_{\La}$. To get the full of $\Gamma_{\rm{Gprj}\mbox{-}A}$, we need only to see how the remaining projective indecomposable modules in $\rm{Gprj}\mbox{-}A$ being $(-, \underline{M}), (-, \underline{N}), (-, \underline{T})$ and $(-, \underline{U})$ can be added into the $\Gamma_{B}$.   It can be easily seen that the minimal right(left) almost morphisms $X \rt Y$ in $\mmod \La$,  where $X \  \text{or} \  Y \in \{M, N, T, U\}$, induce the  minimal right(left) almost morphisms $(-, \underline{X})\rt (-, \underline{Y})$ in $\rm{Gprj}\mbox{-}\underline{mod}\mbox{-}\La$. These minimal right(left) almost morphisms $(-, \underline{X})\rt (-, \underline{Y})$ give us the remaining  part of $\Gamma_{\rm{Gprj}\mbox{-}A}$. Now, we deduce  $\Gamma_{\rm{Gprj}\mbox{-}A}$ looks as follows:
 		$$
 	
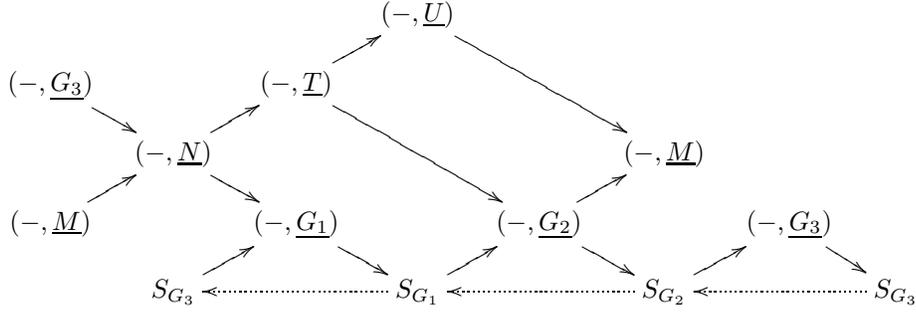
\xymatrix@-5mm{	&&&&(-, \underline{U})\ar[ddrr]&&&&\\&(-, \underline{G_3})\ar[dr]&&(-, \underline{T})\ar[ddrr]\ar[ur]&&&&&\\&&(-, \underline{N})\ar[dr]\ar[ur]&&&&(-, \underline{M})&&\\	
 	&(-, \underline{M})\ar[ur]&	&(-, \underline{G_1})\ar[dr]&&(-, \underline{G_2})\ar[dr]\ar[ur]&&(-, \underline{G_3})\ar[dr]\\ 		&&S_{G_3}\ar[ur]&&S_{G_1}\ar[ur]\ar@{.>}[ll]&&S_{G_2}\ar[ur]\ar@{.>}[ll]&&\ar@{.>}[ll]S_{G_3}	}
 	$$
 	Let us continue our computation in order to find the Auslander-Reiten quiver $\Gamma_{\rm{Gprj}\mbox{-}T_2(\La)}$ of the subcategory of Gorenstein projective modules over $T_2(\La)$. Analogue to Example \ref{Example1} we use the same notation to show   the indecomposable objects in $\rm{Gprj}\mbox{-}T_2(\La)$. Also  we need to find the minimal projective resolutions in $\mmod \rm{Gprj}\mbox{-}\La$ of the indecomposable functors in $\mmod \underline{\rm{Gprj}}\mbox{-}\La$. To do this, the almost split sequences in $\rm{Gprj}\mbox{-}\La$ help us to compute the minimal projective resolutions of the simple functors. We will do as follows: 
 	$$0 \rt (-, G_1)\rt (-, G_2)\rt (-, G_3)\rt S_{G_3}\rt 0$$
 	$$0 \rt (-, G_3)\rt (-, P_1\oplus G_1)\rt (-, G_2)\rt S_{G_2}\rt 0$$
 	$$0 \rt (-, G_2)\rt (-, G_3\oplus P_2)\rt (-, G_1)\rt S_{G_1}\rt 0$$
 	$$0 \rt (-, G_3)\rt (-, P_1)\rt (-, G_3)\rt (-, \underline{G_3})\rt 0$$
 	$$0 \rt (-, G_2)\rt (-, P_1\oplus P_2)\rt (-, G_2)\rt (-, \underline{G_2})\rt 0$$
 	$$0 \rt (-, G_1)\rt (-, P_2)\rt (-, G_1)\rt (-, \underline{G_1})\rt 0.$$
 	By Theorem \ref{main5}, 
 	 we just need to add to the quiver $\Gamma_{B}$ the vertices corresponding to the remaining indecomposable objects. This in view of Lemma \ref{AlmostSplittrivialmonomorphisms} is done as follows.
 	
 		$$
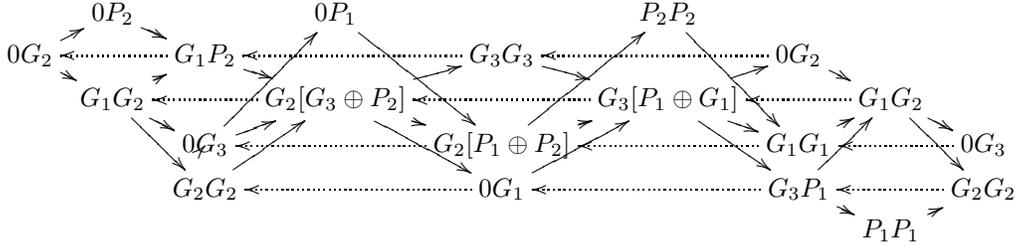
\xymatrix@R=0.05cm@C=0.15cm{ &&0P_2\ar[dr]&& 0P_1\ar[dddr]& &P_2P_2\ar[dddr] & &
 		& & & & \\&0G_2\ar[dr]\ar[ur]&
 		&G_1P_2\ar[dr]\ar@{.>}[ll] & &
 		G_3G_3\ar[dr]\ar@{.>}[ll] & &
 		0G_2\ar[dr]\ar@{.>}[ll]
 		& & 
 		& &
 		&	& \\
 		&&G_1G_2\ar[ur]\ar[ddr]\ar[dr]	&& G_2[G_3\oplus P_2]\ar[dr]\ar[ur]\ar[ddr]\ar@{.>}[ll] & &
 		G_3[P_1\oplus G_1]\ar[ddr]\ar[dr]\ar[ur]\ar@{.>}[ll] & &
 		G_1G_2\ar[ddr]\ar[dr]\ar@{.>}[ll]
 		& & & &
 		\\&&& 0G_3\ar[ur]\ar[uuur]\ar
 		&&
 		G_2[P_1\oplus P_2]\ar[ur]\ar[uuur]\ar@{.>}[ll] & &
 		G_1G_1\ar[ur]\ar@{.>}[ll] & &0G_3\ar@{.>}[ll]
 		& &
 		& \\&& & G_2G_2\ar[uur]
 		& & 0G_1\ar[uur]\ar@{.>}[ll]&
 		&G_3P_1\ar[uur]\ar[dr]\ar@{.>}[ll] &
 		&G_2G_2\ar@{.>}[ll] &
 		& &
 		\\&& &
 		& &&		
 		& &P_1P_1\ar[ur]  & &		
 		& & }$$
  	
\end{example}
 We end up this paper by this fundamental question: As our results show that for an Artin algebra of finite representation type $\La$, one can realize the subcategories of Gorenstein projective modules over $\rm{Gprj}\mbox{-}T_2(\La)$ and $\rm{Gprj}\mbox{-}\rm{Aus}(\underline{\rm{mod}}\mbox{-}\La)$ with the modules category $\mmod \rm{Aus}(\underline{\rm{Gprj}}\mbox{-}\La)$. This question may arise here  whether it is possible to realize the subcategory of Gorenstein projective modules of any Artin algebra with the modules category    of some Artin algebra.

\section{acknowledgment}
The author gratefully thank to the referee for the constructive comments and recommendations
which definitely help to improve the readability and quality of the paper. He would like to thank Hideto Asashiba for some useful comments to improve the English of the manuscript.


\begin{thebibliography}{9999}
	

	
	\bibitem[Au1]{Au1} {\sc	M. Auslander,} {\sl Representation theory of artin algebras I,} Comm. Algebra {\bf 1} (1974),  177-268.
	\bibitem[Au2]{Au2} {\sc M. Auslander,} {\sl Functors and morphisms determined by objects,} Representation theory of algebras (Proc. Conf., Temple Univ., Philadelphia, Pa., 1976), pp. 1–244. Lecture Notes in Pure Appl. Math., Vol. 37, Dekker, New York, 1978. 
	
	\bibitem[AB]{AB} {\sc M. Auslander and M. Bridger,} {\sl Stable module theory,} Mem. Amer. Math. Soc. {\bf 94} (1969).
	\bibitem[AR1]{AR1} {\sc  M. Auslander and I. Reiten,} {\sl Representation theory of artin algebras IV,} Comm. Algebra {\bf 5} (1977), 443-518.
	\bibitem[AR2]{AR2} {\sc M. Auslander and I. Reiten,} {\sl  Applications of contravariantly finite subcategories,} Adv. Math. {\bf 86} (1991) 111-152.
	\bibitem[AR3]{AR3} {\sc M. Auslander and I. Reiten,} {\sl Cohen-Macaulay and Gorenstein algebras,} Progr. Math. {\bf 95} (1991) 221-245.
	
	\bibitem [AS]{AS} {\sc M. Auslander and S.O. Smal{\o},} {\sl Almost split sequences in subcategories,} J. Algebra {\bf 69}(2) (1981), 426-454.
	
\bibitem [B]{B} {\sc A. Beligiannis,} {\sl	On the Freyd Categories of an Additive Category,}   Homology, Homotopy and Appl. {\bf 2}(2000),  147-185.
	\bibitem[CSZ]{CSZ} {\sc X.W. Chen, D. Shen and G. Zhou,} {\sl The Gorentein-projective modules over a monomial algebra,} Proceedings of the Royal Society of Edinburgh Section A: Mathematics, {\bf 148A} (2018), 1115-1134.
\bibitem[DR]{DR} {\sc	V. Dlab and C. M. Ringel,} {\sl The module theoretical approach to quasi-hereditary algebras,}  In: Representations of Algebras and Related Topics (ed.  H. Tachikawa and S. Brenner).  London Math.  Soc.  Lecture Note Series {\bf 168}.  Cambridge University Press (1992), 200-224.
	
	\bibitem[E]{E} {\sc \"{O}. Eiriksson,} {\sl From submodule categories to the stable Auslander algebra,} J. Algebra {\bf 486} (2017), 98-118.
	\bibitem[EJ]{EJ} {\sc  E. Enochs and O. Jenda,} {\sl Gorenstein injective and projective modules,} Math. Z. {\bf 220} (1995), no. 4, 611-633.
	
	\bibitem[EHS]{EHS} {\sc H. Eshraghi, R. Hafezi and Sh. Salarian,} {\sl Total acyclicity for complexes of representations of quivers,} Comm. Algebra { \bf 41} (2013), no. 12, 4425-4441.
	\bibitem[F]{F} {\sc P. Freyd,} {\sl  Representations in abelian categories,}  In Proc. Conf. Categorical Algebra (La Jolla, Calif., 1965), pages 95-120. Springer, New York, 1966.	
	 \bibitem[GLS]{GLS} {\sc C. Geiss, B. Leclerc and J. Schr\"{o}er,}{ \sl  Quivers with relations for symmetrizable Cartan matrices I: Foundations,} Invent. Math. {\bf 209} (2017), no. 1, 61-158.
	\bibitem[H]{H1} {\sc R. Hafezi,} {\sl Auslander-Reiten duality for subcategories,} available on  arXiv:1705.06684.
	
			\bibitem[Ho]{Ho} {\sc H. Holm,} {\sl  Gorenstein homological dimensions,} J. Pure Appl. Algebra {\bf 189} (2004), no. 1-3, 167-193.
		
	
	
\bibitem[KLM]{KLM}	{\sc D. Kussin, H. Lenzing and H. Meltzer,} {\sl Nilpotent  operators and weighted  projective lines,} J. Reine Angew. Math. {\bf 685} (2013), 33-71	
		\bibitem[K]{K} {\sc H. Krause,} {\sl  Krull-Schmidt categories and projective covers,} Expo. Math. {\bf 33} (2015), no. 4, 535-549.
\bibitem[KS]{KS}{\sc 	H. Krause and {\o}. Solberg,} {\sl Applications of cotorsion pairs,} J. London Math.Soc. (2) {\bf 68} (2003), 631-650.	

	\bibitem[LZ1]{LZ1} {\sc X-H. Luo and P. Zhang,} {\sl  Monic representations and Gorenstein-projective modules,} Pacific J. Math. {\bf 264} (2013), no. 1, 163-194.
	\bibitem[LZ2]{LZ2} {\sc X-H. Luo and P. Zhang,} {\sl  Separated monic representations I: Gorenstein-projective modules,} J. Algebra {\bf 479} (2017), 1-34.
	
	\bibitem[MT]{MT} {\sc H. Matsu and R. Takahashi,} {\sl Singularity categories and singular equivalences for resolving subcategories,} Math. Z. {\bf 285} (2017), no. 1-2, 251-286.		
	\bibitem[RS1]{RS1} {\sc C.M. Ringel and M. Schmidmeier,}  {\sl Invariant subspaces of nilpotent linear operators I,} J. Reine Angew. Math. {\bf 614} (2008), 1-52. 
	\bibitem[RS2]{RS2} {\sc C.M. Ringel and M. Schmidmeier,} {\sl The Auslander-Reiten translation in submodule categories,} Trans. Amer. Math. Soc. {\bf 360} (2008), no. 2, 691-716.
	
	\bibitem[RZ1]{RZ1} {\sc  C.M. Ringel and P. Zhang,} {\sl  From submodule categories to preprojective algebras,}  Math. Z. {\bf 278} (2014), no. 1-2, 55-73. 
	\bibitem[RZ2]{RZ2} {\sc C.M. Ringel and P. Zhang,} {\sl Representations  of quivers over  the algebras of dual  numbers,} J. Algebra. {\bf 475} (2017), 327-360.
	\bibitem[RZ3]{RZ3} {\sc C.M. Ringel and P. Zhang,} {\sl Gorenstein-projective and semi-Gorenstein-projective modules,}  Algebra Number Theory {\bf 14} (2020), no. 1, 1-36. 
	\bibitem[RZ4]{RZ4} {\sc C.M. Ringel and P. Zhang,} {\sl Gorenstein-projective and semi-Gorenstein-projective modules II,}  J. Pure Appl. Algebra {\bf 224} (2020), no. 6,  23 pp. 
		
	\bibitem[XZZ]{XZZ}{\sc B-L. Xiong, P. Zhang and  Y-H. Zhang,} {\sl Auslander-Reiten translations in monomorphism categories,} Forum Math. {\bf 26} (2014), no. 3, 863-912.
	\bibitem[ZX]{XZ}{\sc P. Zhang  and B-L. Xiong,}  {\sl Separated monic representations II: Frobenius subcategories and RSS equivalences,} Trans. Amer. Math. Soc. {\bf 372} (2019), no. 2, 981-1021.
\end{thebibliography}
\end{document}